\documentclass{amsart}
\usepackage{mathrsfs}
\usepackage{amssymb,latexsym,amsmath,amsthm,enumitem,mathrsfs}
\usepackage{lmodern}
\usepackage{hyperref}
\usepackage{color}
\usepackage{stmaryrd}
\usepackage{mathrsfs}
\usepackage{lineno}
\usepackage{bbm}
\usepackage{scalerel,stackengine}
\usepackage{tikz}
\usepackage{hyperref}
\usepackage{enumitem}
\usepackage{esint}% See geometry.pdf to learn the layout options. There are lots.
\usepackage{graphicx}
\usepackage{amssymb}
\newtheorem{ingredient}{Ingredient}
\newtheorem{theorem}{Theorem}[section]

\newtheorem{lemma}[theorem]{Lemma}

\newtheorem{remark}[theorem]{Remark}
\newtheorem{proposition}[theorem]{Proposition}

\newtheorem{definition}[theorem]{Definition}
\newtheorem{claim}{Claim}

\newtheorem{corollary}[theorem]{Corollary}

\theoremstyle{definition}

\newcommand{\cD}{\mathcal{D}}

\newcommand{\bF}{\mathbf{F}}

\newcommand{\Q}{\mathbb Q}
\newcommand{\Z}{\mathbb Z}
\newcommand{\C}{\mathbb C}
\newcommand{\N}{\mathbb N}
\newcommand{\T}{\mathbb T}

\newcommand{\e}{\varepsilon}

\def\set4{\mathcal I}
\def\tup14{(1,2,3,4)}

\newcommand{\supp}{\mathrm{supp}}

\newcommand{\R}{\mathbb{R}}

\newcommand{\de}{\delta} %%%%%\delta
\newcommand{\De}{\Delta} %%%%%\delta
\newcommand{\ga}{\gamma}

\newcommand{\wh}{\widehat}
\newcommand{\wt}{\widetilde}
\newcommand{\si}{\sigma}
\newcommand{\Si}{\Sigma}

\newcommand{\I}{\mathbb I}
\newcommand{\dist}{\textup{dist}}

\newcommand{\Id}{\boldsymbol 1}

\newcommand{\al}{\alpha}

\newcommand{\gd}{\textup{good}}
\newcommand{\bd}{\textup{bad}}
\newcommand{\om}{\omega}

\newcommand{\RD}{\textup{RapDec}}
\newcommand{\dimH}{\dim_{\textup{H}}}
\newcommand{\dimP}{\dim_{\textup{P}}}
\newcommand{\dimB}{\overline{\dim_{\textup{B}}}}
\newcommand{\io}{\iota}
\newcommand{\err}{\textup{error}}
\newcommand{\BR}{{\rm{Br}}}

\begin{document}

 \author{Shengwen Gan}
 \address{Department of Mathematics\\
 Sun Yat-sen University, Guangzhou, 510275, China}
  \email{shengwengan2018@gmail.com}

\author{Bochen Liu}
\address{Department of Mathematics \& International Center for Mathematics, Southern University of Science and Technology, Shenzhen 518055, China}
\email{Bochen.Liu1989@gmail.com}

  \author{Shukun Wu}
 \address{Department of Mathematics, Indiana University Bloomington, USA}
  \email{shukwu@iu.edu}


\date{}
\thanks{Bochen Liu was partially supported by the National Key R\&D Program of China 2024YFA1015400. The draft of this paper was completed in March 2026. The authors acknowledge the use of generative AI tools for proof checking and typo correction.}

\title[Falconer's distance set problem]{Lebesgue measure of distance sets and $L^2$ ball inflation}

\begin{abstract}
We give an alternative proof of the $5/4$ result of Guth-Iosevich-Ou-Wang on the Falconer distance conjecture: if a compact set $E$ in the plane has Hausdorff dimension $>5/4$, then there exists $x\in E$ such that the pinned distance set
$$\Delta_x(E):=\{|x-y|: y\in E\} $$
has positive Lebesgue measure. Moreover, we are able to beat the $5/4$ barrier if the packing dimension of $E$ is $<2$. More generally, we obtain sufficient conditions involving both the Hausdorff and packing dimensions of $E$. These conditions improve the $5/4$ threshold when the packing dimension of $E$ is $<2$, and recover the second author's result for $\dimH E=\dimP E>1$. The proof is based on $L^2$ ball inflation and does not use decoupling inequalities.
\end{abstract}

\maketitle

\section{Introduction}
For a set $E\subset \R^2$, define its distance set to be
\[ \De(E):=\{|x-x'|:x,x'\in E\}. \]
We also define the pinned distance set:
For $p\in \R^2$, define
\[ \De_p(E):=\{|x-p|:x\in E\}. \]
Let $\dimH$ denote the Hausdorff dimension, and $\dimP$ denote the packing dimension.

The goal of this paper is to prove the following result.

\begin{theorem}\label{maintheorem}
    Let $E\subset\R^2$ be a compact set and let $\al=\dimH(E)>1, \beta=\dimP(E)$. If 
    \begin{equation}\label{cond1}
        \al>\frac32-\frac{1}{2\beta}
    \end{equation} and 
    \begin{equation}\label{cond2}
        P(\al,\beta):=\al^3-(2\beta-1)\al^2-2\al+(3\beta-1)<0,
    \end{equation}
    then there exists a point $x\in E$ such that the pinned distance set $\{ |x-y| \}_{y\in E}$ has positive Lebesgue measure.
\end{theorem}

As a corollary, we recover the result of Guth, Iosevich, Ou and Wang \cite{guth2020falconer}.
\begin{corollary}
    If $E\subset \R^2$ is compact and satisfies 
    \[ \dimH(E)>\frac54, \]
    then there exists a point $x\in E$ such that the pinned distance set $\{ |x-y| \}_{y\in E}$ has positive Lebesgue measure.
\end{corollary}
\begin{proof}
Denote $\al:=\dimH(E)>\frac54, \beta:=\dimP(E)$.
    Since $\beta\le 2$, we have
    \[ \al>\frac54\ge \frac32-\frac{1}{2\beta}. \]
Noting that the coefficient of $\beta$ in $P(\al,\beta)$ is $3-2\al^2<0$ and $\beta\in [\al,2]$, we have
\[ P(\al,\beta)\le P(\al,\al)=-\al^3+\al^2+\al-1=-(\al-1)^2(\al+1)<0. \] 
By Theorem \ref{maintheorem}, the corollary holds.
\end{proof}
\begin{remark}
    {\rm
    It is surprising to us that our proof does not use decoupling, but nevertheless reaches the ``$\frac54$-threshold". 
    }
\end{remark}

We also recover the equal Hausdorff/packing dimension case proved by the second author \cite{liu2026lebesgue}.

\begin{corollary}
    If $E\subset \R^2$ is compact and satisfies \[\dimH(E)=\dimP(E)>1,\] then there exists a point $x\in E$ such that the pinned distance set $\{ |x-y| \}_{y\in E}$ has positive Lebesgue measure.
\end{corollary}
\begin{proof}
    Denote $\al:=\dimH(E)=\dimP(E)>1$. Then $\al>\frac32-\frac{1}{2\al}$, and $P(\al,\al)=-(\al-1)^2(\al+1)<0$. By Theorem \ref{maintheorem}, the corollary holds.
\end{proof}

If one wants to improve the $\frac54$-threshold, one observes a specific structure in the set $E$:
\begin{corollary}
    If $E\subset \R^2$ is a compact set satisfying $\dimH(E)=\frac54 $
    and $|\De_x(E)|=0$ for all $x\in E$, then $\dimP(E)=2$.
\end{corollary}
\begin{proof}
    Suppose $\beta:=\dimP(E)<2$, then $\al:=\dimH(E)=\frac54>\frac32-\frac{1}{2\beta}$. By Theorem \ref{maintheorem}, we must have
    \[ P(\frac54,\beta)\ge 0. \]
    However, $P(\frac54,\beta)\le P(\frac54,\frac54)<0$, which is a contradiction.
\end{proof}

\subsection{Historical results}

Falconer \cite{falconer1985hausdorff} proved that if $E\subset\mathbb R^d$ is compact and $\dim_{\mathrm
H}(E)>(d+1)/2$, then $|\Delta(E)|>0$. He conjectured that the conclusion should hold under the hypothesis
$\dim_{\mathrm H}(E)>d/2$. In the plane, Wolff \cite{wolff1999decay} proved the result for $\dim_{\mathrm H}(E)>4/3$.
In higher dimensions, Erdogan \cite{erdogan2005bilinear} proved it for $\dim_{\mathrm H}(E)>d/2+1/3$. Later, Du, Guth,
Ou, Wang, Wilson, and Zhang \cite{du2021weighted}, and Du and Zhang \cite{MR3961084}, developed weighted $L^2$
restriction estimates and used them to obtain further progress on Falconer's distance problem.

In 2019, the second author \cite{liu20192} proposed an $L^2$ approach to the pinned distance set problem. A
breakthrough was subsequently obtained by Guth, Iosevich, Ou, and Wang \cite{guth2020falconer}, who proved that if
$E\subset\mathbb R^2$ is compact and $\dim_{\mathrm H}(E)>5/4$, then $|\Delta_x(E)|>0$ for some $x\in E$. Their proof
introduced the good-bad decomposition and applied decoupling inequalities to the distance set problem. This remains
the best known threshold in dimension two.

The other direction on Falconer's distance problem is to give a lower bound on the Hausdorff dimension of $\De(E), \De_x(E)$ assuming $\dim_H (E)>\frac{d}{2}$. 
In the plane, the best record was obtained by Stull \cite{stull2022pinned}, who used computability theory to show that
$\dimH(E)>1$ implies $\dimH(\De_x(E))\ge \frac34$ for some $x\in E$. Also, see the work of Katz and Tao \cite{katz2001some}, Bourgain \cite{bourgain2003erdHos}, Wang and Zheng \cite{wang2022improvement}, Shmerkin \cite{shmerkin2023non}, Keleti and Shmerkin \cite{keleti2019new}, Shmerkin and Wang \cite{shmerkin2025distance}, Fiedler and Stull \cite{fiedler2023dimension} \cite{fiedler2024pinned}. In \cite{fiedler2023dimension}, Fiedler and Stull studied the distance problem taking into account both Hausdorff dimension and packing dimension. They proved that for $E\subset \R^2$ with $\al=\dimH(E)>1, \beta=\dimP(E)$, if $\beta<\frac{(3+\sqrt{5})\al-1-\sqrt{5}}{2}$, then $\dim(\De_x(E))=1$ for some $x\in E$. Comparing with Theorem \ref{maintheorem}, on the one hand, our conclusion is $|\De_x(E)|>0$; on the other hand, we obtain some new range for the pair $(\al,\beta)$ (for example, $(\al,\beta)=(\frac54,2-\e)$ is in our range but not in their range).

The distance set problem is also studied under some regularity assumption. See the work of Orponen \cite{orponen2012distance} \cite{orponen2017distance}, Shmerkin \cite{shmerkin2019hausdorff} \cite{shmerkin2021improved}.
Very recently, the second author \cite{liu2026lebesgue} proved that  $\dimH(E)=\dimP(E)>1$ implies $|\De_x(E)|>0$ for some $x\in E$. He introduced a multiscale good-bad decomposition and a multiscale Mizohata-Takeuchi-type estimate.

\subsection{Structure of the paper}
In Section~\ref{sec2}, we sketch the proof of the main theorem. In Section~\ref{sec3}, we introduce the terminology
used throughout the paper. In Section~\ref{sec4}, we prove a lemma for Lipschitz functions. In Section~\ref{sec5}, we
reduce Theorem~\ref{maintheorem} to Proposition~\ref{reducedprop}. In Section~\ref{sec6}, we perform the multiscale
good-bad decomposition and estimate the bad part. In Section~\ref{sec7}, we estimate the good part. In
Section~\ref{sec8}, we prove Theorem~\ref{maintheorem}. Finally, in Section~\ref{secweightedL2}, we prove the weighted
$L^2$ estimate.

\subsection{Notation} We use $A\lesssim B$ to denote $A\le CB$, where $C$ is some universal constant. If $A\lesssim B$ and $B\lesssim A$, then we denote it by $A\sim B$. For a fixed scale $R\ge 1$, we use $A\lessapprox B$ to denote that $A\le C_\ga R^\ga B$ for any $\ga>0$. Similarly, we use $A\approx B$ to denote $A\lessapprox B$ and $B\lessapprox A$. For $R\ge 1$, we use $\RD(R)$ to denote a quantity that $\lesssim C_N R^{-N}$ for any $N\ge 1$.

% \subsection{Acknowledgement.}

\section{An overview of the proof}\label{sec2}
\subsection{Weighted \texorpdfstring{$L^2$}{} estimate}
The key task in the paper is to prove a weighted $L^2$ estimate, which we formulate as follows. 

Fix $u\in [1,2]$. Let $X\subset B(0,1)$ be a union of disjoint $R^{-1}$-balls. We assume $X$ satisfies:
\begin{equation}\label{moraldimH}
\begin{split}
    &|X\cap B(x,r)|\lesssim R^{-2}(rR)^u \ \ \ \textup{for~any~}x\in B(0,1)\textup{~and~}R^{-1}\le r\le 1,\\
    &|X|= R^{-2}R^u.
\end{split}
\end{equation}
(Here, $|E|$ is the Lebesgue measure of the set $E$.)
Suppose $F:\R^2\rightarrow \mathbb C$ is any function with $\supp \wh F\subset S^1_R$, where $S^1_R$ is the circle of radius $R$. We are interested in finding the smallest number $D(u)$ such that
\begin{equation}\label{keyest}
    \frac{1}{|X|}\|F\|_{L^2(X)}^2\lessapprox R^{D(u)} \|F\|_{L^2(B(0,1))}^2
\end{equation}
holds for all $R\ge 1$ and $X, F$ as above. We use $\bF(u)$ to denote the following statement.
\medskip

\fbox{$\bF(u)$}: \textit{If $E\subset \R^2$ is a compact set of Hausdorff dimension greater than $u$, then there exists $x\in E$ such that $\De_x(E)$ has positive Lebesgue measure.}
\medskip

To help understand the ideas, we will make several informal claims in this section without proofs. The rigorous proofs will be given in later sections.

We have the following implication from \eqref{keyest} to Falconer's distance problem.
\begin{claim}
    Suppose \eqref{keyest} holds for $D(u)$ and $D(u)\le u-1$, then $\bF(u)$ holds.
\end{claim}

\begin{remark}
{\rm 
A classical weighted $L^2$ estimate shows that \eqref{keyest} holds for $D(u)=1-\frac{u}{2}$ (see, for example, \cite[Corollary 1.7]{MR3961084}). Plugging into $D(u)\le u-1$, we see that $\bF(\frac{4}{3})$ holds. This corresponds to the result of Wolff \cite{wolff1999decay}.}
\end{remark}

\medskip

It was observed by Guth-Iosevich-Ou-Wang \cite{guth2020falconer} that one does not need the full strength of \eqref{keyest} in order to show $\bF(u)$. Indeed, one may throw away the ``bad" part of $F$ which can be handled by Orponen's radial projection theorem \cite{orponen2018dimension}, and prove an estimate for the remaining ``good" part. Such an estimate suffices to derive $\bF(u)$. The trick nowadays is known as the good-bad decomposition. Recently, the second author \cite{liu2026lebesgue} introduced a multiscale good-bad decomposition, which we will adopt in this paper. Since it is too complicated to state here, we postpone the definition later. At present, we just use it as a black box and assume for any function $F$ as above, we have a multiscale good-bad decomposition
\[ F=F_\gd+F_\bd. \]
Here, $F_\gd$ is obtained from $F$ by throwing away some ``heavy" wave packets. So, we can still assume $\supp\wh F_\gd\subset S^1_R$ and $\|F_\gd\|_2\lesssim \|F\|_2$.

Similar to \eqref{keyest}, we are now interested in
\begin{equation}\label{keyest1}
    \frac{1}{|X|}\|F_\gd\|_{L^2(X)}^2\lessapprox R^{D_\gd(u)}  \|F_\gd\|_{L^2(B(0,1))}^2.
\end{equation}
We have a similar implication:
\begin{claim}\label{prop2.3}
    Suppose \eqref{keyest1} holds for $D_\gd(u)$ and $D_\gd(u)\le u-1$, then $\bF(u)$ holds.
\end{claim}

\begin{remark}
{\rm Guth-Iosevich-Ou-Wang \cite{guth2020falconer} proved that \eqref{keyest1} holds for $D_\gd(u)=\frac{2-u}{3}$. Plugging into $D_\gd(u)\le u-1$, we see that $\bF(\frac{5}{4})$ holds. 
Recently, the second author \cite{liu2026lebesgue} showed that $D_\gd(u)=0$ under the assumption that $u>1$ and $X$ is AD-regular. (Here, AD-regular means $|X\cap B(x,r)|\sim R^{-2}(rR)^u$ for all $x\in X$ and $R^{-1}\le r\le 1$.)}
\end{remark}

\subsection{Packing dimensions}
To prove Theorem \ref{maintheorem}, we also need to incorporate the packing dimension. Recall that \eqref{moraldimH} is a discretized interpretation that $X$ has Hausdorff dimension $\ge u$. Now we introduce a discretized interpretation about the packing dimension.

Let $|X|_r$ denote the smallest number of $r$-balls needed to cover $X$, also called the $r$-covering number of $X$. Let $\kappa\in [u,2]$. In addition to \eqref{moraldimH}, we also require $X$ to satisfy:
\begin{equation}\label{moraldimP}
    |X|_r\lesssim r^{-\kappa} \textup{~for~all~}R^{-1}\le r\le 1.
\end{equation}
We are interested in the following estimate:
\begin{equation}\label{keyest2}
    \frac{1}{|X|}\|F_\gd\|_{L^2(X)}^2\lessapprox R^{D_\gd(u,\kappa)} \|F_\gd\|_{L^2(B(0,1))}^2,
\end{equation}
where $X$ satisfies \eqref{moraldimH} and \eqref{moraldimP}. Denote the following statement:

\fbox{$\bF(u,\kappa)$}: \textit{If $E\subset \R^2$ is a compact set of Hausdorff dimension greater than $u$ and packing dimension smaller than $\kappa$, then there exists $x\in E$ such that $\De_x(E)$ has positive Lebesgue measure.}
\medskip

We have the following informal results.
\begin{claim}
    Suppose \eqref{keyest2} holds for $D_\gd(u,\kappa)$ and $D_\gd(u,\kappa)\le u-1$, then $\bF(u,\kappa)$ holds.
\end{claim}

\begin{claim}\label{claim4}
    \eqref{keyest2} holds for $D_\gd(u,\kappa)=\max\{\frac{\kappa-u}{2\kappa-1}, \frac{\kappa-u}{(u+1)(2\kappa-u-1)}\}$. Therefore, by the previous Claim, $\bF(u,\kappa)$ holds if $u>\frac32-\frac{1}{2\kappa}$ and $P(u,\kappa)<0$. 
\end{claim}

\medskip

\subsection{Ball inflation process}\label{subsecballinflation}
We discuss the strategy to prove Claim \ref{claim4}. It is through a ball inflation process. Recall that $X$ is a union of $R^{-1}$-balls satisfying \eqref{moraldimH} and \eqref{moraldimP}. To pass from the left-hand side to the right-hand side in \eqref{keyest2}, we choose a sequence of scales $1=r_0>r_1>\dots>r_{k+1}=R^{-1}$. Let $N_{r_i}(X)$ be the $r_i$-neighborhood of $X$. We may morally view $N_{r_i}(X)$ as a union of $r_i$-balls. Essentially, $N_{r_{k+1}}(X)$ is $X$ and $N_{r_0}(X)$ is $B(0,1)$. For each $1\le i\le k$, we are going to prove an estimate
\begin{equation}\label{onestep}
    \frac{1}{|N_{r_i}(X)|}\|F_\gd\|_{L^2(N_{r_i}(X))}^2\lessapprox R^{D_i}\frac{1}{|N_{r_{i-1}}(X)|}\|F_\gd\|_{L^2(N_{r_{i-1}}(X))}^2. 
\end{equation} 
Each time, the balls are inflated from radius $r_i$ to $r_{i-1}$. Combining all the estimates gives us
\[ \frac{1}{|X|}\|F_\gd\|_{L^2(X)}^2\lessapprox R^{D_1+\dots+D_{k+1}}  \|F_\gd\|_{L^2(B(0,1))}^2. \]
Hence, our goal becomes to carefully choose a sequence $\{r_i\}$ so as to minimize $\sum_{i=1}^{k+1} D_i$.

Next, we will focus on the one-step ball inflation estimate \eqref{onestep}. Before doing that we introduce some notation. First of all, we may assume $X$ satisfies the \textbf{uniform property}: 
For any $R^{-1}\le r\le 1$, the quantities $|B_r\cap X|$ are comparable for all dyadic $r$-cubes $B_r$ such that $B_r\cap X\neq \emptyset$. 

By the uniform property of $X$, one may see that $|X|_r\sim \frac{|X|}{|B_r\cap X|}$.
We define the \textbf{branching function} associated to $X$ to be $f:[0,1]\rightarrow [0,\infty)$ in the following way. For $x\in [0,1]$, let $r=R^{-x}$. Then $f(x)$ is defined so that 
\[|X|_r\sim R^{f(x)}.\] 
(Note that $f$ is defined up to an error about $O(\frac{1}{\log R})$.)

We summarize the properties of $f$.
\begin{enumerate}[label=(\roman*)]
    \item $f$ is increasing; $f(0)=0$ since $|X|_1=1$; $f(1)=u$ since $|X|_{R^{-1}}=R^{2}|X|=R^u$.
    \item $f(x)\ge u x$, since $R^{f(x)}=|X|_r\sim\frac{|X|}{|B_r\cap X|}\gtrsim \frac{R^{-2}R^u}{R^{-2}(rR)^u}=R^{ux}$, where $r=R^{-x}$.
    \item $\textup{Lip}(f)\le 2$. Indeed, for $x'>x$, $r=R^{-x},r'=R^{-x'}$, one has $R^{f(x')-f(x)}\sim \frac{|X|_{r'}}{|X|_r}\sim \frac{|B_r\cap X|}{|B_{r'}\cap X|}\lesssim \frac{|B_r|}{|B_{r'}|}\sim R^{2(x'-x)}$.
    \item $f(x)\le \kappa  x$. Indeed, for $r=R^{-x}$, one has $R^{f(x)}=|X|_{r}\lesssim r^{-\kappa}=R^{x\kappa}$. Here, we used \eqref{moraldimP}.
\end{enumerate}

Now we give the following definitions.
\begin{definition}\label{defFKT}
    Let $f$ be the branching function associated to $X$. For $[x_1,x_2]\subset [0,1]$ and $s\in[0,2]$, we say $f$ is $s$-\textbf{Frostman} on $[x_1,x_2]$ if for all $x\in [x_1,x_2]$,
    \[ f(x)\ge f(x_1)+s(x-x_1). \]
    We say $f$ is $s$-\textbf{Katz-Tao} on $[x_1,x_2]$ if for all $x\in [x_1,x_2]$,
    \[ f(x)\ge f(x_2)+s(x-x_2). \]
\end{definition}
We state the three main ingredients that are key steps in the proof of Claim \ref{claim4}. The first one was introduced by the second author in \cite{liu2026lebesgue}.

\begin{ingredient}[Lossless ball inflation (informal)]\label{lem1}
    Let $f$ be the branching function associated to $X$. Let $[x_1,x_2]\subset [0,1]$ and $r_i=R^{-x_i}$ $(i=1,2)$. If $f$ is $1$-Frostman on $[x_1,x_2]$ and $1-x_1\le 2(1-x_2)$, then
    \begin{equation}
         \frac{1}{|N_{r_2}(X)|}\|F_\gd\|_{L^2(N_{r_2}(X))}^2\lessapprox \frac{1}{|N_{r_1}(X)|}\|F_\gd\|_{L^2(N_{r_1}(X))}^2.
    \end{equation}
\end{ingredient}

\begin{ingredient}\label{lem2}
    Let $f$ be the branching function associated to $X$. Let $x\in [0,1]$ and $r=R^{-x}$. If $f$ is $u$-Katz-Tao on $[x,1]$, then
    \begin{equation}
        \|F_\gd\|_{L^2(X)}^2\lessapprox (\frac{|X|_{R^{-1}}}{|X|_r})^{\frac{u}{u+1}} \frac{1}{rR}\|F_\gd\|_{L^2(N_{r}(X))}^2.
    \end{equation}
\end{ingredient}
\begin{remark}
    {\rm
    Our weighted $L^2$ estimate is different from that proved by Du and Zhang \cite{MR3961084}. Comparing with \cite[(1.11)]{MR3961084}, our weighted $L^2$ inequality contains a factor $|X|_{R^{-1}}$ on the right hand side. Therefore, we are able to obtain some gain when the size of $X$ is small.}
\end{remark}

\begin{ingredient}[A structure lemma for Lipschitz functions]\label{lem3}
    Suppose {$1<u<\kappa\leq 2$}. Let $f:[0,1]\rightarrow [0,\infty)$ be a piecewise linear Lipschitz function satisfying: $f$ is increasing, $f(0)=0,f(1)=u$, $ux\le f(x)\le \kappa x$, $\textup{Lip}(f)\le 2$. Then there exists a sequence $0=x_0<x_1<\dots<x_{k}<x_{k+1}=1$ such that $f$ is $1$-Frostman on $[x_{i-1},x_i]$ and $1-x_{i-1}\le 2(1-x_i)$ for $1\le i\le k$ and $f(x_k)=(2\kappa-1)x_{k}-\kappa+1$. (Geometrically, $(x_{k},f(x_{k}))$ lies on the line segment $\overline{IF}$ in Figure \ref{Fig1}.)
\end{ingredient}

\begin{remark}
{\rm 
    We imposed a technical assumption: $f$ is piecewise linear. Indeed, without loss of generality, one may always assume $f$ is piecewise linear by perturbing the original $f$. The piecewise linear assumption is for the following use: In the proof of Ingredient \ref{lem3}, an algorithm is carried out to extract a sequence $\{x_i\}$. This assumption guarantees that the algorithm will stop in finite steps.}
\end{remark}

We are ready to give the proof of Claim \ref{claim4}.
\begin{proof}[Proof of Claim \ref{claim4}]
    {Assume first that $u<\kappa$.} Using Ingredient \ref{lem3}, we obtain a sequence $0=x_0<x_1<\dots<x_{k}<x_{k+1}=1$. By Ingredient \ref{lem1} and Ingredient \ref{lem2}, we have
    \begin{equation}
        \frac{1}{|X|}\|F_\gd\|_{L^2(X)}^2\lessapprox \frac{|N_{r_{k}}(X)|}{|X|}(\frac{|X|_{R^{-1}}}{|X|_{r_{k}}})^{\frac{u}{u+1}}\frac{1}{r_{k}R} \|F_\gd\|_{L^2(B(0,1))}^2, 
    \end{equation} 
where $r_{k}=R^{-x_{k}}$.
It remains to check 
\begin{equation}\label{finalestimate}
    \frac{|N_{r_{k}}(X)|}{|X|}(\frac{|X|_{R^{-1}}}{|X|_{r_{k}}})^{\frac{u}{u+1}}\frac{1}{r_{k}R}\lesssim R^{\max\{\frac{\kappa-u}{2\kappa-1}, \frac{\kappa-u}{(u+1)(2\kappa-u-1)} \}}. 
\end{equation} 
Note that $|N_{r_{k}}(X)|\sim r_{k}^2|X|_{r_{k}}$, $|X|_{r_{k}}\sim R^{f(x_{k})}$, $|X|\sim R^{-2}|X|_{R^{-1}}\sim R^{-2}R^u$. The left-hand side in the above inequality is
\[ Rr_{k}(\frac{|X|_{r_{k}}}{|X|_{R^{-1}}})^{\frac{1}{u+1}}\sim R^{1-x_{k}+\frac{f(x_{k})-u}{u+1}}. \]
Using $f(x_{k})=(2\kappa-1)x_{k}-\kappa+1$, we see that the above quantity is 
\[R^{1-x_k+\frac{(2\kappa-1)x_k-\kappa+1-u}{u+1}}.\] 
It suffices to check
\begin{equation}\label{tocheck}
    1-x_k+\frac{(2\kappa-1)x_k-\kappa+1-u}{u+1}\le \max\{\frac{\kappa-u}{2\kappa-1}, \frac{\kappa-u}{(u+1)(2\kappa-u-1)}\}. 
\end{equation} 
This inequality is linear in $x_k$. Noting that $ux_k\le f(x_k)\le u$, we obtain
\[ \frac{\kappa-1}{2\kappa-u-1}\le x_k\le \frac{u+\kappa-1}{2\kappa-1}. \]
Plugging the two endpoints of $x_k$ will finish the proof {when $u<\kappa$}.

{If $1<u=\kappa<2$, apply this argument with any $\kappa'\in(u,2]$: the resulting exponent tends to $0$ as $\kappa'\downarrow u$, giving the claim in the $\lessapprox$ sense. If $u=\kappa=2$, the claim follows directly from $X\subset B(0,1)$ and $|X|=1$.}
\end{proof}

\subsection{Parameters} In the paper, we will use parameters $\e,\eta, T, \de$ to handle some technicalities. The relation between them is 
\begin{equation}\label{choiceofpar}
    \de\ll T^{-1}\ll \eta\ll \e \ll 1.
\end{equation}
A precise choice will be given later.

\section{Uniform measures and branching functions}\label{sec3}

In this section, we show that for a given measure, we can decompose it as a sum of uniform measures plus a negligible term. We also define the branching function for a uniform measure.

For a number $0<\rho\le 1$, we use $\cD_\rho=\{Q_\rho\}$ to denote the set of dyadic $\rho$-cubes. In other words, each cube in $\cD_\rho$ is of form $[n\rho,(n+1)\rho)\times [m\rho,(m+1)\rho)$ for some $n,m\in \Z$. Throughout the paper, we fix a large integer $T$, which will be determined later.

\begin{definition}
    Given a sequence $w=(w_1,\dots,w_T)\in [0,1]^T$ and $R\gg 1$, we say a measure $\nu$ is $(w;T)$-\textbf{uniform} at scale $R$ if for any $1\le i\le T$ and any $Q_{R^{-\frac{i}{T}}}\in \cD_{R^{-\frac{i}{T}}}$ with $Q_{R^{-\frac{i}{T}}}\cap \supp\nu\neq\emptyset$, we have
    \[ \nu(Q_{R^{-\frac{i}{T}}})\in [w_i,2w_i). \]
\end{definition}

\begin{lemma}\label{decomplemma}
Let $0<\de<1$.
    Let $\nu$ be a probability measure supported in $[0,1)^2$. Let $R\gg 1$. Then there exist $N\lesssim (\log R)^{2T}$ measures $\nu^1,\dots,\nu^{N+1}$ such that: 
    \begin{enumerate}
        \item $\sum_{i=1}^{N+1}\nu^i = \nu$;
        \item $\nu^i([0,1)^2)\ge R^{-\de^2}$ for $1\le i\le N$, and  $\nu^{N+1}([0,1)^2)\le R^{-\de^3}$;
        \item For $1\le i\le N$, $\nu^i$ is $(w^i;T)$-uniform for some $w^i\in [0,1]^T$.
    \end{enumerate} 
\end{lemma}
 
To prove Lemma \ref{decomplemma}, we use the following one-step refinement.

\begin{lemma}\label{onesteplem}
    Let $\nu$ be a measure supported in $[0,1)^2$ and suppose $1\ge \nu([0,1)^2)\ge R^{-\de^3}$. Let $R\gg 1$. Then there exists a measure $\nu^1$ such that: $\nu^1\le \nu$; $\nu^1$ is $(w^1;T)$-uniform for some $w^1=(w^1_1,\dots,w^1_T)$; and 
    \[ \nu^1([0,1)^2)\ge \frac{1}{(C\log R)^T}\nu([0,1)^2).\]
    Here, $C$ is a universal constant. If $R$ is sufficiently large, then $\nu^1([0,1)^2)\ge \frac{1}{(C\log R)^T}R^{-\de^3}\ge R^{-\de^2}$.
\end{lemma}
\begin{proof}[Proof of Lemma \ref{onesteplem}]
    The proof follows from a standard pigeonhole argument. We first sort $Q\in \cD_{R^{-1}}$ according to $\nu(Q)$. Let 
    \[\cD_{R^{-1},w}:=\{ Q\in\cD_{R^{-1}}: \nu(Q)\in [w,2w) \},\]
    \[ \cD_{R^{-1},\textup{tiny}}:=\{Q\in\cD_{R^{-1}}: \nu(Q)<R^{-3}\}. \]
    Then 
    \[ \cD_{R^{-1}}=\cD_{R^{-1},\textup{tiny}}\sqcup\bigsqcup_{w\textup{dyadic},w\in [R^{-3},1]}\cD_{R^{-1},w}.  \]
Note that 
\[ \sum_{Q\in \cD_{R^{-1},\textup{tiny}}}\nu(Q)\le R^{-3}R^2\le\frac{1}{2}\nu([0,1)^2). \]
By pigeonholing, there exists $w_1\in [R^{-3},1]$, such that 
\[ \sum_{Q\in \cD_{R^{-1},w_1}}\nu(Q)\gtrsim \frac{1}{\log R}\nu([0,1)^2). \]
We let $\nu_1=\nu|_X$, where $X={\cup_{Q\in\cD_{R^{-1},w_1}}Q}$. Then
\[ \nu_1([0,1)^2)\gtrsim \frac{1}{\log R}\nu([0,1)^2). \]
Iterating this argument $T$ steps, we obtain $w^1=(w_1,\dots,w_T)$ and a $w^1$-uniform measure $\nu^1$, such that
\[ \nu^1([0,1)^2)\ge \frac{1}{(C\log R)^T}\nu([0,1)^2). \]

\end{proof}

\begin{proof}[Proof of Lemma \ref{decomplemma}]
    Applying Lemma \ref{onesteplem}, we obtain a $(w^1;T)$-uniform measure $\nu^1$ with $\nu^1([0,1)^2)\ge \frac{1}{(C\log R)^T}\nu([0,1)^2).$ Equivalently,
    \[ (\nu-\nu^1)([0,1)^2)\le (1-\frac{1}{(C\log R)^T})\nu([0,1)^2). \]
Applying Lemma \ref{onesteplem} to $\nu-\nu^1$, we obtain a $(w^2;T)$-uniform measure $\nu^2$ satisfying
\[ (\nu-\nu^1-\nu^2)([0,1)^2)\le (1-\frac{1}{(C\log R)^T})^2\nu([0,1)^2). \]
Iterating this process $N$ steps, we obtain $\nu^1,\dots,\nu^{N}$. Here, $N$ is chosen so that $\nu^{N+1}:=\nu-\nu^1-\dots-\nu^{N}$ satisfies
\[ \nu^{N+1}([0,1)^2)\le R^{-\de^2} \]
for the first time.

Since $\nu^{N+1}([0,1)^2)\le (1-\frac{1}{(C\log R)^T})^N\nu([0,1)^2)=(1-\frac{1}{(C\log R)^T})^N$ and noting that $(1-\frac{1}{(C\log R)^T})^{C_1(\log R)^{2T}}<R^{-\de^2}$ for large $C_1$, we have
\[ N\lesssim (\log R)^{2T}. \]
\end{proof}

\medskip

Next, we define the branching function of a uniform measure. Recall that $\nu$ is $(w;T)$-uniform at scale $R$ if for each $0\le i\le T$, $\nu$ is supported in a union of dyadic $R^{-\frac{i}{T}}$-cubes, such that the mass of $\nu$ in each such cube is $\sim w_i$. Recall that $|X|_\rho$ is the $\rho$-covering number of $X$. We may allow the measure to have total mass $<1$.

\begin{definition}[Branching function]
    Let $\nu$ be a measure that is $(w;T)$-uniform at scale $R$, and is supported in $[0,1)^2$. Define $f:[0,1]\rightarrow [0,\infty)$ as follows.

    For $i\in \{0,1,\dots,T\}$, define $f(\frac{i}{T})$ so that
    \[ |\supp\nu|_{R^{-\frac{i}{T}}}=R^{f(\frac{i}{T})}. \]
Then define $f(x)$ by linearly interpolating the values $\{f(\frac{i}{T})\}$.
    
\end{definition}

\begin{remark}
    {\rm
    To be more flexible, we may define $f$ in the way that
    \[ |\supp\nu|_{R^{-\frac{i}{T}}}\sim R^{f(\frac{i}{T})}. \]
    In this sense, the branching function is determined up to an error $O(\frac{1}{\log R})$, which is allowable for our use if $R$ is big enough.
    
    }
\end{remark}

It is easy to see $\nu([0,1)^2)\sim R^{f(\frac{i}{T})}w_i\sim R^{f(\frac{i}{T})}\nu (Q_{R^{-\frac iT}})$ for each $i$. For each $x\in [0,1]$, since $x$ is $\frac 1T$-close to some $\frac iT$ and $f$ is $2$-Lipschitz, we have
\begin{equation}\label{verifybranching}
    \frac{\nu(Q_{R^{-x}})}{\nu([0,1)^2)}\in [R^{-\frac 2T-f(x)}, R^{\frac 2T-f(x)}]
\end{equation}
whenever $Q_{R^{-x}}\cap \supp\nu\neq\emptyset$. If $T$ is large enough, then one may essentially view 
\[\nu(Q_{R^{-x}})\sim \nu([0,1)^2)R^{-f(x)}.\]

We recall the definition of Frostman/Katz-Tao conditions.
\begin{definition}\label{FKTdef}
    Let $f:[0,1]\rightarrow[0,\infty)$ be a Lipschitz function. For $[x_1,x_2]\subset [0,1]$ and $s\in[0,2]$, we say $f$ is $s$-\textbf{Frostman} on $[x_1,x_2]$ if for all $x\in [x_1,x_2]$,
    \[ f(x)\ge f(x_1)+s(x-x_1). \]
    We say $f$ is $s$-\textbf{Katz-Tao} on $[x_1,x_2]$ if for all $x\in [x_1,x_2]$,
    \[ f(x)\ge f(x_2)+s(x-x_2). \]
    
    For $0<\ga\ll s$,
we say $f$ is $s$-\textbf{Frostman} with error $\ga$ on $[x_1,x_2]$ if for all $x\in [x_1,x_2]$,
    \[ f(x)\ge f(x_1)+s(x-x_1)-\ga. \]
    We say $f$ is $s$-\textbf{Katz-Tao} with error $\ga$ on $[x_1,x_2]$ if for all $x\in [x_1,x_2]$,
    \[ f(x)\ge f(x_2)+s(x-x_2)-\ga. \]
    
If $f$ is $s$-Frostman (resp. $s$-Katz-Tao) with error $\ga$ on $[0,1]$, we simply say $f$ is $s$-Frostman (resp. $s$-Katz-Tao) with error $\ga$.
    
    For a measure $\nu$ that is $(w;T)$-uniform at scale $R$, we say $\nu$ is $s$-Frostman (resp. $s$-Katz-Tao) if its branching function is $s$-Frostman (resp. $s$-Katz-Tao).    
\end{definition}

The next lemma tells us that we can partition an $s$-Frostman measure into several measures that are both $s$-Frostman and $s$-Katz-Tao, up to small errors.

\begin{lemma}\label{partitionlem}
Suppose $0<\ga\ll s$.
    Let $\nu$ be a measure that is $(w;T)$-uniform at scale $R$ and is $s$-Frostman with error $\ga$. Then there are $M\sim \frac{|\supp\nu|_{R^{-1}}}{R^{s-C\ga}} $ measures $\nu^{(1)},\dots,\nu^{(M)}$ such that each $\nu^{(j)}$ is $s$-Frostman up to an exponent error $C\ga+O(1/\log R)$ with uniform comparability constant, and satisfies $|\supp \nu^{(j)}|_{R^{-1}}\sim R^{s-C\ga}$. Also, $\nu=\sum_{j=1}^M \nu^{(j)}$. Moreover, in the comparability convention of the preceding remark, all $\nu^{(j)}$ admit a common branching profile $\wt f$ with $\wt f(\frac{i}{T})=\min\{f(\frac{i}{T}),s-C\ga\}$ and $|\supp\nu^{(j)}|_{R^{-i/T}}\sim R^{\wt f(i/T)}$. Here, $C$ is a universal constant.
\end{lemma}

\begin{proof}
    Let $X$ be the union of cubes in $\cD_{R^{-1}}$ that intersect $\supp\nu$. Let $f$ be the branching function of $\nu$. Then 
    \[ |X|_{R^{-x_i}}\sim R^{f(x_i)} \]
    for $0\le i\le T$. Here we let $x_i=\frac{i}{T}$. That $f$ is $s$-Frostman with error $\ga$ means $f(x)\ge s x-\ga$. We need to modify $f$ so that its value at $x=1$ is $s-C\ga$.

    More precisely, our goal is to find a partition $X=\bigsqcup_{j=1}^M X^{(j)}$ such that for each $j$,
    \[ |X^{(j)}|_{R^{-x_i}}\sim R^{\wt f(x_i)} \]
    for $0\le i\le T$. Then we just set $\nu^{(j)}=\nu|_{X^{(j)}}$ and finish the proof.

    Next, we construct these $X^{(j)}$. First, let $i$ be the minimal index such that $|X|_{R^{-x_i}}>R^{s-C\ga}$. For each $Q'\in \cD_{R^{-x_{i-1}}}$ such that $Q'\cap X\neq \emptyset$, we equally divide $X\cap Q'$ into $A\sim R^{f(x_i)-s+C\ga}$ pieces 
    \[X\cap Q'= \bigsqcup_{i=1}^A X_i(Q') ,\]
    each of which is a union of about $\frac{|Q'\cap X|_{R^{-x_i}}}{R^{f(x_i)-s+C\ga}}$ dyadic $R^{-x_i}$-cubes.
    This is allowable since 
    \[\frac{|Q'\cap X|_{R^{-x_i}}}{R^{f(x_i)-s+C\ga}}\sim \frac{|X|_{R^{-x_i}}/|X|_{R^{-x_{i-1}}}}{R^{f(x_i)-s+C\ga}}\sim \frac{R^{s-C\ga}}{|X|_{R^{-x_{i-1}}}}\ge 1.\]
    The last inequality is because of the minimality of $i$.
    
    Now let $X_i=\bigcup_{Q'}X_i(Q')$. One can check the branching function of each $X_i$ is 
    \[ g(x)=\begin{cases}
        f(x) & x\le x_{i-1}\\
        f(x)-f(x_i)+s-C\ga & x\ge x_i,
    \end{cases} \]
    and the value of $g$ in $(x_{i-1},x_i)$ is by linear interpolation at the two endpoints.

Then we repeat the process for each $X_i$ whose branching function is $g$. After finitely many steps, we obtain $\{X^{(j)}\}$ with the common branching profile $\wt f$ in the above comparability sense.
    
\end{proof}

\section{A structure lemma for Lipschitz functions}\label{sec4}

In this section, we prove a structure lemma for Lipschitz functions. 
For convenience, we make the following definitions.
\begin{definition}[Admissible interval]\label{defadmissible}
    For $[a,b]\subset[0,1]$, we say $[a,b]$ is admissible if $1-a\le 2(1-b)$.
\end{definition}

\begin{definition}
    Let $f:[0,1]\rightarrow \R$. If there exists $0=a_0<a_1<\dots<a_M=1$ such that $f|_[a_{i-1},a_i]$ is linear for all $1\le i\le M$, then we say $f$ is piecewise linear with $M$ pieces. Each $[a_{i-1},a_i]$ is called an $f$-piece.
\end{definition}

The main result is as follows.

\begin{lemma}\label{preprestructurelem}
    Let $1<u\le \kappa\le 2$. Then there exists $\e(u,\kappa)>0$ such that for any $0<\e<\e(u,\kappa)$, the following holds. Let $f:[0,1]\rightarrow [0,\infty)$ be a piecewise linear function, satisfying: $f$ is increasing, $f(0)=0$, $f(1)=u$, $ux\le f(x)\le \kappa x$, $\textup{Lip}(f)\le 2$. 
    Then there exists a sequence $0=x_0<x_1<\dots<x_{k}<x_{k+1}=1$ such that  
    \begin{enumerate}
        \item For $1\le i\le {k}$, $f$ is $(1+\e)$-Frostman on $[x_{i-1},x_i]$ and $[x_{i-1},x_i]$ is admissible.
        \item $f(x_{k})=(2\kappa-1)x_{k}-\kappa+1+O(\e)$.
        \item At most $C(\e,u,\kappa)$ of the selected endpoints $x_0,\dots,x_k$ are not endpoints of the given $f$-pieces. In particular, $k\le \#\{f\textup{-pieces}\}+C(\e,u,\kappa)$ for some $C(\e,u,\kappa)$.
    \end{enumerate}
\end{lemma}

\begin{figure}[ht]
\centering
\begin{minipage}[b]{0.85\linewidth}
\includegraphics[width=9cm]{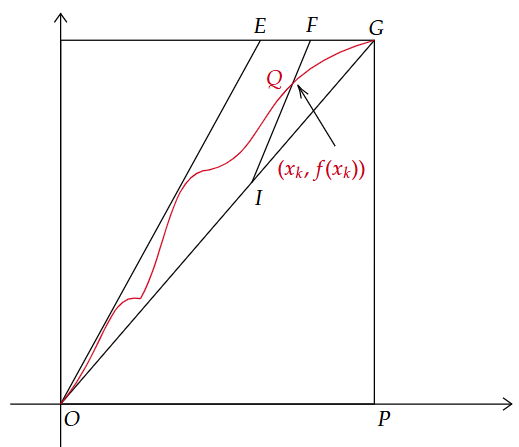}
\caption{}
\label{Fig1}
\end{minipage}
\end{figure}

\begin{proof}

We first draw a graph as in Figure \ref{Fig1}. The coordinates of marked points are: 
\begin{enumerate}
    \item $O=(0,0)$;
    \item $P=(1,0)$;
    \item $E=(\frac{u}{\kappa},u)$;
    \item $F'=(\frac{u+\kappa-1}{2\kappa-1},u)$;
    \item $G=(1,u)$;
    \item $I'=(\frac{\kappa-1}{2\kappa-u-1},\frac{u(\kappa-1)}{2\kappa-u-1})$;
    
\end{enumerate}

One can also check the following property. Given any point $Y$ on the line segment $\overline{F'I'}$ (here $\overline{AB}$ means the line segment with endpoints $A,B$), draw a line through $Y$ with slope $1$. Suppose this line intersects lines $\overline{OE}$ and $\overline{GP}$ at $Y_1$ and $Y_2$, then $Y$ is the midpoint of $\overline{Y_1Y_2}$. Indeed, one only needs to check for $Y=F',I'$. One can also check that the line through $F',I'$ is
\[ y=(2\kappa-1)x-\kappa+1. \]

Take $0<\e<(u-1)/4$ and put $s=1+\e$, 
$$h(x):=f(x)-sx.$$
Let $q$ be the first intersection of the graph of $f$ with the line
\[
L(x):=(2\kappa-s)x-\kappa+s+\e.
\]
Since $L(0)<f(0)$ and $L(1)>f(1)$, we have $0<q<1$. Moreover,
\begin{equation}
\label{displayed-id-f}
	f(q)=(2\kappa-1)q-\kappa+1+\e(2-q)
\end{equation}
and
$$1-q\ge\frac{\kappa-u+\e}{2\kappa-s}\ge\frac{\e}{3}.$$

Starting with $b=q$, put $a_0=\max\{0,2b-1\}$ and let $a$ be the leftmost minimizer of $h$ on $[a_0,b]$. 
\begin{itemize}
	\item If $a_0=0$, then $h(b)\ge(u-s)b>h(0)$;
	\item  If $a_0>0$, the upper bound $f(x)\leq\kappa x$ and the choice of $q$ give $$h(a_0)\le(\kappa-s)a_0<h(b).$$

\end{itemize}
In particular, $a\neq b$, thus $a<b$. Then the minimum property on the choice of $a$ shows that $f$ is $s$-Frostman on $[a,b]$, and $a\ge a_0\geq 2b-1$ makes this interval admissible. Repeat with $b=a$ until reaching $0$, then list the endpoints in increasing order, with $x_k=q$.

Since $h$ is piecewise linear, each new left endpoint is either an endpoint of a given $f$-piece or $a=a_0=2b-1$. In the latter case the interval has length $1-b\ge1-q\ge\e/3$. These intervals have disjoint interiors, so there are at most $3/\e+1$. This proves finite termination and (3). 

The displayed identity \eqref{displayed-id-f} for $f(q)$ proves (2), that completes the proof.
\end{proof}

As a remark, the preceding construction only uses the upper bound $f(x)\le\kappa x$ at $a_0$. Choose $\e_0=(u-1)/(4(3-u))$. If $0<a_0<\e_0$, then $b=(1+a_0)/2$, and the Lipschitz and lower bounds give instead
\[
h(a_0)\le(2-s)a_0<\frac{u-s}{2}(1+a_0)\le h(b).
\]
Thus the same construction gives the following stronger version.

\begin{lemma}\label{prestructurelem}
    Let $1<u\le \kappa\le 2$. Then there exists $\e_0=\e_0(u,\kappa)>0$ and $\e(u,\kappa)>0$ such that for any $0<\e<\e(u,\kappa)$, the following holds. Let $f:[0,1]\rightarrow [0,\infty)$ be a piecewise linear function, satisfying: $f$ is increasing, $f(0)=0$, $f(1)=u$, $f(x)\ge ux$ for $x\in[0,1]$, $f(x)\le \kappa x$ for $x\in [\e_0,1]$, $\textup{Lip}(f)\le 2$. Then there exists a sequence $0=x_0<x_1<\dots<x_{k}<x_{k+1}=1$ such that 
    \begin{enumerate}
        \item For $1\le i\le {k}$, $f$ is $(1+\e)$-Frostman on $[x_{i-1},x_i]$ and $[x_{i-1},x_i]$ is admissible.
        \item $f(x_{k})=(2\kappa-1)x_{k}-\kappa+1+O(\e)$.
        \item At most $C(\e,u,\kappa)$ of the selected endpoints $x_0,\dots,x_k$ are not endpoints of the given $f$-pieces. In particular, $k\le \#\{f\textup{-pieces}\}+C(\e,u,\kappa)$ for some $C(\e,u,\kappa)$.
    \end{enumerate}
\end{lemma}

We need a more quantitative version of Lemma \ref{prestructurelem}, which takes care of the number of linear pieces for a Lipschitz function. We give the following definition which is the same as \cite[Definition 4.2]{shmerkin2023non}.

\begin{definition}
    We say $f$ is \textbf{linear} with error $\e$ on $[x_1,x_2]$ if for all $x\in [x_1,x_2]$,
     \[ |f(x)-f(x_1)-s_f(x_1,x_2)(x-x_1)|\le \e(x_2-x_1). \]
    Here, $s_f(x_1,x_2)=\frac{f(x_2)-f(x_1)}{x_2-x_1}$.
\end{definition}
We recall \cite[Lemma 4.4]{shmerkin2023non}.
\begin{lemma}\label{linearapproxlem}
    For every $\eta>0$, there is $\tau=\tau(\eta)>0$ such that the following holds: for any $2$-Lipschitz function $f:[0,1]\rightarrow \R$, there exists a family of non-overlapping intervals $\{[c_j,d_j]\}$ such that:
    \begin{enumerate}
        \item $f$ is linear on $[c_j,d_j]$ with error $\eta$.
        \item $d_j-c_j\ge \tau$.
        \item $|[0,1]\setminus \bigcup_j [c_j,d_j]|\le \eta$.
    \end{enumerate}
\end{lemma}

Now, we state the structure lemma.
\begin{lemma}\label{structlem}
    Suppose $1<u\le\kappa\le 2$. Then there exists $\e_0=\e_0(u,\kappa)>0$, $\e(u,\kappa)>0$. For any $0<\e<\e(u,\kappa)$, there exists $\eta(\e,u,\kappa)>0$. For any $0<\e<\e(u,\kappa)$ and $0<\eta<\eta(\e,u,\kappa)$, there exists $\tau=\tau(\eta)>0$ such that the following holds. Let $f:[0,1]\rightarrow [0,\infty)$ be a Lipschitz function, satisfying: $f$ is increasing, $f(0)=0$, $f(1)=u$, $f(x)\ge ux$ for $x\in[0,1]$, $f(x)\le \kappa x$ for $x\in [\e_0,1]$, $\textup{Lip}(f)\le 2$. Then there exist $k\le 3\tau^{-1}+C(\e,u,\kappa)$ and $k$ disjoint intervals $\{[a_j,b_j]\}_{j=1}^k$ such that 
    \begin{enumerate}
        \item For $1\le j\le {k}$, $[a_j,b_j]$ is admissible and
        \[ f(x)\ge f(a_j)+(1+\e)(x-a_j)-\eta(b_j-a_j)\ \ \ \ x\in[a_j,b_j]. \]
        \item $|[0,b_k]\setminus \bigcup_{j=1}^k [a_j,b_j]|\le \eta$.
        \item $f(b_{k})=(2\kappa-1)b_{k}-\kappa+1+O(\e)$.
    \end{enumerate}
Moreover, the intervals may be chosen with $b_j-a_j\ge\ell_0(\e,\eta,u,\kappa)>0$.
\end{lemma}
\begin{proof}
    First, we apply Lemma \ref{linearapproxlem} to $f$ and with $\eta$ replaced by $\eta^4$, to obtain intervals $\{[c_l,d_l]\}_{l=1}^L$ so that the three properties in Lemma \ref{linearapproxlem} hold:
     \begin{enumerate}
        \item $f$ is linear on $[c_l,d_l]$ with error $\eta^2$. (If it is with error $\eta^4$, then it is with error $\eta^2$.)
        \item $d_l-c_l\ge \tau=\tau(\eta)$.
        \item $|[0,1]\setminus \bigcup_l [c_l,d_l]|\le \eta^4$.
    \end{enumerate}

    We define a new function $g:[0,1]\rightarrow \R$ that is the piecewise-linear interpolation of the data $(0,f(0)), (c_l,f(c_l)),(d_l,f(d_l)) (1\le l\le L), (1,f(1))$, with $(\e_0,f(\e_0))$ also included. Thus $g$ has at most $2\tau^{-1}+2$ linear pieces and satisfies the upper bound on $[\e_0,1]$. Also, since $g$ is defined by linear interpolation, the original $\eta^4$ approximation implies
    \[ |f(x)-g(x)|\le \eta^2(d_l-c_l) \]
when $x\in[c_l,d_l]$. Also, $\|f-g\|_\infty\le2\eta^4$ because the total length of the gaps is at most $\eta^4$.

    Applying Lemma \ref{prestructurelem} to $g$, we obtain a sequence $0=x_0<x_1<\dots<x_m<x_{m+1}=1$ so that 
    \begin{enumerate}
        \item For $1\le i\le m$, $[x_{i-1},x_i]$ is admissible and $g$ is $(1+\e)$-Frostman on $[x_{i-1},x_i]$.
        \item $g(x_m)=(2\kappa-1)x_m-\kappa+1+O(\e)$.
        \item At most $C(\e,u,\kappa)$ endpoints are not among the interpolation nodes. 
    \end{enumerate}  
    Let $\I_1$ be the collection of $[x_{i-1},x_i]$ such that both $x_{i-1}$ and $x_i$ belong to $\{0,1,c_l,d_l:1\le l\le L\}$; let $\I_2:=\{[x_{i-1},x_i]\}_{i=1}^m\setminus \I_1$. Then $\#\I_2\le M$ for some $M=M(\e,u,\kappa)\ge1$, independent of $f$ and $\eta$. 
    
    Now, we are ready to define $\{[a_j,b_j]\}_{j=1}^k$. Denote $I_b:=[0,1]\setminus \bigcup_l[c_l,d_l]$. Then
    \[|I_b|\le \eta^4.\]
    Define $\I'_1$ to be those $[x_{i-1},x_i]\in\I_1$ such that 
    \[|I_b\cap [x_{i-1},x_i]|\le \eta^2(x_{i}-x_{i-1}).\] 
    Define $\I'_2$ to be those $[x_{i-1},x_i]\in \I_2$ such that $x_i-x_{i-1}\ge\eta/(4M)$. Finally, we define $\{[a_j,b_j]\}_{j=1}^k$ to be $\I_1'\bigcup \I_2'$.

    We verify (1) in the lemma. For convenience, we set $d_{-1}=0$. By definition, each $[a_j,b_j]$ is contained in some $[x_{i-1},x_i]$, so $[a_j,b_j]$ is admissible. We need to verify for any $x\in[a_j,b_j]$ that
    \[ f(x)\ge f(a_j)+(1+\e)(x-a_j)-\eta(b_j-a_j). \]

    Consider the case $[a_j,b_j]\in \I'_1$. Then $[a_j,b_j]$ is the union of intervals $[c_l,d_l]$ and $[d_{l-1},c_l]$. For example, if $a_j=c_s,b_j=d_t$ for some $s\le t$, then
    \[ [a_j,b_j]=[c_s,d_s]\bigcup [d_s,c_{s+1}]\bigcup\dots\bigcup[c_t,d_t]. \]
 If $x\in [c_r,d_r]\subset [a_j,b_j]$, then by the assumption that $f$ is linear on $[c_r,d_r]$ with error $\eta^2$, we have
 \[ f(x)\ge g(x)-\eta^2(d_r-c_r). \]
By the assumption that $g$ is $(1+\e)$-Frostman on $[a_j,b_j]$, we have
\[ g(x)\ge g(a_j)+(1+\e)(x-a_j). \]
Combining the two estimates and noting $g(a_j)=f(a_j)$, we obtain
\[ f(x)\ge f(a_j)+(1+\e)(x-a_j)-\eta^2(b_j-a_j). \]

If $x\in [d_{r-1},c_r]\subset [a_j,b_j]$, then by the previous discussion (noting $d_{r-1}\in [c_{r-1},d_{r-1}]$), we have
\[ f(d_{r-1})=g(d_{r-1})\ge f(a_j)+(1+\e)(d_{r-1}-a_j). \]
Since $[a_j,b_j]\in \I'_1$, we have $(c_r-d_{r-1})\le \eta^2(b_j-a_j)$. Therefore using that $f$ is increasing and $1+\e\le2$, we have
\[
f(x)\ge f(d_{r-1})\ge f(a_j)+(1+\e)(x-a_j)-2\eta^2(b_j-a_j).
\]
Assuming $\eta<1/2$, we finish this case.

\medskip

Now, consider the case $[a_j,b_j]\in\I'_2$. Since $b_j-a_j\ge\eta/(4M)$ and $\|f-g\|_\infty\le2\eta^4$, we have, for $x\in[a_j,b_j]$,
\begin{equation}
\begin{aligned}
f(x)&\ge f(a_j)+(1+\e)(x-a_j)-4\eta^4\\
&\ge f(a_j)+(1+\e)(x-a_j)-\eta(b_j-a_j),
\end{aligned}
\end{equation}
provided $16M\eta^2\le1$.

\bigskip
Next, we verify (2) of the lemma. Note that $b_k\le x_m$. We will actually prove a stronger estimate:
\begin{equation}\label{justprove}
    |[0,x_m]\setminus \bigcup_{j=1}^k[a_j,b_j]|\le \eta. 
\end{equation} 
Recall $\{[a_j,b_j]\}_{j=1}^k=\I_1'\bigcup \I'_2$ and $[0,x_m]=\bigcup_{I\in \I_1\cup\I_2} I$. We just need to verify
\[ \sum_{I\in \I_1\setminus \I'_1}|I|+\sum_{I\in \I_2\setminus \I'_2}|I|\le \eta. \]
If $I\in \I_1\setminus \I'_1$, then $|I_b\cap I|>\eta^2|I|$. Hence, we get
\[ \sum_{I\in \I_1\setminus \I'_1}|I|\le \eta^{-2}|I_b|\le \eta^2. \]
By definition of $\I'_2$, we have
\[ \sum_{I\in \I_2\setminus \I'_2}|I|\le M\frac{\eta}{4M}=\frac{\eta}{4}. \]
Combining the two estimates, we finish the proof of (2) in the lemma. Every retained interval in $\I_1$ contains a complete approximation interval, so all retained intervals have length at least $\ell_0=\min\{\tau(\eta),\eta/(4M)\}>0$.

\bigskip
Finally, we show (3) in the lemma. By \eqref{justprove}, we have $|x_m-b_k|\le \eta$. Since \[g(x_m)=(2\kappa-1)x_m-\kappa+1+O(\e),\]  we obtain
\[ g(b_k)=(2\kappa-1)b_k-\kappa+1+O(\e+\eta). \]
The uniform approximation gives
\[
|f(b_k)-g(b_k)|\le2\eta^4.
\]
Letting $\eta<\eta(\e,u,\kappa)\le \e$, we obtain
\[ f(b_k)=(2\kappa-1)b_k-\kappa+1+O(\e). \]

\end{proof}

For some technicalities, we need the following version, in which the condition $f(x)\ge ux$ is modified to $f(x)\ge ux-\de_0$, and the condition $f(1)=u$ is modified to $f(1)\in [u-\de_0,u]$. If $\de_0$ is much smaller than other parameters, one can expect that the structure lemma still holds. 

\begin{lemma}[Structure lemma]\label{structurelem}
    Suppose $1<u\le\kappa\le 2$. Then there exists $\e_0=\e_0(u,\kappa)>0$ and $\e(u,\kappa)>0$. For any $0<\e<\e(u,\kappa)$, there exists $\eta(\e,u,\kappa)>0$. For any $0<\e<\e(u,\kappa)$ and $0<\eta<\eta(\e,u,\kappa)$, there exists $\tau=\tau(\eta)>0$ and $\de_0=\de_0(\eta,\e,u,\kappa)>0$ such that the following holds. Let $f:[0,1]\rightarrow [0,\infty)$ be a Lipschitz function, satisfying: $f$ is increasing, $f(0)=0$, $f(1)\in [u-\de_0,u]$, $f(x)\ge ux-\de_0$ for $x\in[0,1]$, $f(x)\le \kappa x$ for $x\in [\e_0,1]$, $\textup{Lip}(f)\le 2$. Then there exist $k\le 3\tau^{-1}+C(\e,u,\kappa)$ and $k$ disjoint intervals $\{[a_j,b_j]\}_{j=1}^k$ such that 
    \begin{enumerate}
        \item For $1\le j\le {k}$, $[a_j,b_j]$ is admissible and
        \begin{equation}\label{modifiedf}
            f(x)\ge f(a_j)+(1+\e)(x-a_j)-\eta(b_j-a_j)\ \ \ \ x\in[a_j,b_j]. 
        \end{equation} 
        \item $|[0,b_k]\setminus \bigcup_{j=1}^k [a_j,b_j]|\le \eta$.
        \item $f(b_{k})=(2\kappa-1)b_{k}-\kappa+1+O(\e)$.
    \end{enumerate}
\end{lemma}
\begin{proof}
We briefly explain how to prove it. Apply Lemma \ref{structlem} with $f$ replaced by $\wt f(x)=\max\{f(x),ux\}$, for which $0\le\wt f-f\le\de_0$ and $\eta$ replaced by $\frac12 \eta$. Instead of \eqref{modifiedf}, we have
\[ f(x)+\de_0 \ge f(a_j)+(1+\e)(x-a_j)-\frac12\eta(b_j-a_j). \]
By Lemma \ref{structlem}, $b_j-a_j\ge\ell_0(\e,\eta/2,u,\kappa)$. Taking $\de_0\le\min\{\e,\frac12\eta\ell_0(\e,\eta/2,u,\kappa)\}$ absorbs the error and preserves the stopping relation.
\end{proof}

\section{A standard reduction}\label{sec5}
We first make a standard reduction to Theorem \ref{maintheorem}. Recall that $\al=\dimH(E), \beta=\dimP(E)$. Without loss of generality, we assume $E\subset [0,1)^2$ and there exist $E_1,E_2\subset E$ so that $\dist(E_1,E_2)\gtrsim 1$ and $\dimH(E_1)=\dimH(E_2)=\al$. We may decrease $\al$ slightly, so that \eqref{cond1} and \eqref{cond2} still hold. By Frostman's lemma, $E_1$ (resp. $E_2$) admits a probability measure $\mu$ (resp. $\nu$) supported on them satisfying
\begin{equation}
    \mu(B(x,r))\lesssim r^\al\ (\textup{resp.~}\nu(B(x,r))\lesssim r^\al).
\end{equation}

By a classical result in geometric measure theory, packing dimension equals the modified upper box dimension.
Since $\overline{\dim}_{\textup{MB}}(E)=\dimP(E_2)\le \beta$, for any $\ga>0$, there exists a covering $E_2=\bigcup_{i=1}^\infty E_2^i$ with 
\[ \dimB(E_2^i)\le \beta+\ga. \]
Here, $\dimB$ is the upper box dimension. Recall the definition
\[ \dimB(E):=\limsup_{\de\rightarrow 0}\frac{\log |E|_\de}{-\log \de}, \]
\[ \overline{\dim}_{\textup{MB}}(E)=\inf_{E\subset \bigcup_{i=1}^\infty E_j}\sup_j\dimB(E_j).   \]

By pigeonholing, there exists $i_0$ such that $\nu(E_2^{i_0})\gtrsim \frac{1}{i_0^2}$. Now, we replace $E_2$ with $E_2^{i_0}$, and replace $\nu$ with $\frac{1}{\nu(E_2^{i_0})}\nu|_{E_2^{i_0}}$. Since $i_0$ is now fixed, we allow all implicit constants below depending on $i_0$. Then we still have the $\al$-Frostman condition $\nu(B(x,r))\lesssim r^\al$.

Since $\dimB(E_2)\le \beta+\ga$, for any $\ga>0$, there exists $R_\circ\gg 1$ such that for any $R\ge R_\circ$, we have
\[ |E_2|_{R^{-1}}\le R^{\beta+2\ga}. \]
We may replace $\beta$ with $\beta+2\ga$ and let $\ga$ be small enough so that \eqref{cond1} and \eqref{cond2} still hold.

From our reduction, we see that in order to prove Theorem \ref{maintheorem}, we just need to show the following.
\begin{proposition}\label{reducedprop}
Suppose $1< \al\le \beta\le 2$ satisfy
\eqref{cond1} and \eqref{cond2}.
    Let $E_1, E_2\subset [0,1)^2$ with $\dist(E_1,E_2)\gtrsim 1$. Assume $E_1$ (resp. $E_2$) admits a probability measure $\mu$ (resp. $\nu$) supported on them satisfying
\begin{equation}\label{frostcondition}
    \mu(B(x,r))\lesssim r^\al\ (\textup{resp.~}\nu(B(x,r))\lesssim r^\al)
\end{equation}
for all $x\in [0,1)^2$ and $r>0$.
Also assume there exists $R_\circ>1$ such that
\begin{equation}\label{packingcondition}
    |E_2|_{R^{-1}}\le R^\beta
\end{equation}
for all $R\ge R_\circ$.
Then there exists $x\in E_2$ such that $|\De_x(E_1)|>0$.

\end{proposition}

Next, we follow the argument in \cite[Section 2]{guth2020falconer} to perform a wave packet decomposition.
We choose $R_0$ that is much larger than $R_\circ$ in Proposition \ref{reducedprop} (to be determined later). Indeed, we will choose
\begin{equation}\label{R0tbd}
    R_0\ge R_\circ^{CT}, 
\end{equation} 
where $T$ is the parameter we introduced before, $C$ is a large constant.

Let $R_\iota=2^\iota R_0$ ($\iota=1,2,\dots$). For each $\io\ge 1$, cover the unit circle by $R_\io^{-1/2}$-caps $\theta$. For each such $\theta$, find a rectangular block $B_\theta$ with dimensions $\sim R_\io^{1/2}\times R_\io$, with the long direction of each $\tau$ radial. $\{B_\theta\}$ is chosen so that they cover the annulus $R_{\iota-1}\le |\om|\le R_\iota$ (in the frequency space). We choose a partition of unity subordinate to this cover, so that
\begin{equation}\label{def-psi-i-theta}
    1=\psi_0+\sum_{\iota\ge 1}\sum_{\theta: R^{-1/2}_\io\textup{-cap}}\psi_{\iota,\theta}. 
\end{equation} 
% For each $\io$, we define 
% \[ \psi_\io=\sum_\tau \psi_{\io,\tau}, \]
% where the summation is over $\tau$ with dimensions $R_\io^{1/2}\times R_\io$. In other words, $\psi_\tau$ is essentially the bump function on $B_{R_\io}\setminus B_{R_\io/2}$.

Let $\de>0$ be a small number to handle the rapidly decaying tail.
For each $(\iota,\theta)$, cover $[-1,2]^2$ (in the physical space) by tubes $T$ of dimensions $\sim R_\iota^{-1/2+\de}\times 1$ with long axis parallel to the long axis of $B_\theta$. Let $\T_{\iota,\theta}$ be the collection of all these tubes, and let $\eta_T$ be a partition of unity subordinate to this covering, so that for each choice of $\iota$ and $\theta$, $\sum_{T\in \T_{\iota,\theta}}\eta_T=1$ on $B(0,10)$. Denote $\T_\io=\bigcup_\theta \T_{\io,\theta}$.

Define $M_0 f:=(\psi_0 \wh f)^\vee$. 
% For each $\io$, define
% \begin{equation}\label{MT2}
%     M_\io f:= (\psi_\tau \wh f)^\vee.
% \end{equation}
For each $T\in\T_{\iota,\tau}$, define
\begin{equation}\label{MT}
    M_T f:=\eta_T(\psi_{\iota,\theta}\wh f)^\vee.
\end{equation}
Morally speaking, the physical support of $M_T f$ is $T$, while the Fourier support of $M_T f$ is $B_\theta$.

The rest of the paper is devoted to the proof of Proposition \ref{reducedprop}.
Let $\nu$ be the measure as in Proposition \ref{reducedprop}. For $\io\in \{1,2,\dots\}$, apply Lemma \ref{decomplemma} to $\nu$ and scale $R_\io$. Then there exist $N_\io\lesssim (\log R_\io)^{2T}$ measures $\nu^1_\io,\dots,\nu^{N+1}_\io$ such that:
\begin{enumerate}
    \item $\sum_{i=1}^{N_\io+1}\nu_\io^i= \nu$; 
    \item $\nu_\io^{i}([0,1)^2)\ge R_\io^{-\de^2}$ for $1\le i\le N_\io$, and $\nu_\io^{N_\io+1}([0,1)^2)\le R_\io^{-\de^3}$;
    \item For $1\le i\le N_\io$, $\nu^i_\io$ is $(w_\io^i;T)$-uniform for some $w_\io^i\in [0,1]^T$.
\end{enumerate}

Let us just fix a $\nu^i_\io$ ($i\le N_\io$) and denote its branching function by $f_\io^i$. We claim that $f_\io^i$ is $\big(\al-O(\frac{1}{T}+\de^2)\big)$-Frostman. By \eqref{verifybranching}, we have
\[ R_\io^{-\frac{2}{T}-f_\io^i(x)}\le \frac{\nu_\io^i(Q_{R_\io^{-x}})}{\nu_\io^i([0,1)^2)}\le \nu(Q_{R_\io^{-x}})R_\io^{\de^2}\lesssim R_\io^{-x\al}R_\io^{\de^2}, \]
where the last inequality is by \eqref{frostcondition}. We see that
\[ f_\io^i(x)\ge \al x-O(\frac{1}{T}+\de^2)\ge \al x-O(T^{-1}), \]
where the last equality is by \eqref{choiceofpar}.
In other words, $f_\io^i$ is $\al$-Frostman with error $O(T^{-1})$. 

We apply Lemma \ref{partitionlem} with parameters $\ga=O(T^{-1})$, $s=\al$ and $\nu=\nu_\io^i$, to obtain a decomposition
\[ \nu_\io^i=\sum_j \nu_\io^{i,(j)}, \]
where each $\nu_\io^{i,(j)}$ has the branching function $f_\io^{i,(j)}$ that is $\al$-Frostman with error $O(T^{-1})$, and 
\begin{equation}\label{f(1)}
    f_\io^{i,(j)}(1)=\al-C_1 T^{-1}.
\end{equation} 
That $f_\io^{i,(j)}$ is $\al$-Frostman with error $O(T^{-1})$ means
\begin{equation}\label{alphakt}
    f_\io^{i,(j)}(x)\ge \al x- O(T^{-1}).  
\end{equation} 

Since $\supp \nu_\io^{i,(j)}\subset \supp \nu\subset E_2$, by \eqref{packingcondition}, we have
\[ |\supp \nu_\io^{i,(j)} |_{R_\io^{-x}}\le R_\io^{\beta x}, \]
provided $R_\io^x\ge R_0^x \ge R_\circ$.
We choose $R_0$ such that $R_0\ge R_\circ^{CT}$ (recall \eqref{R0tbd}), so we have
\[ f_\io^{i,(j)}(x)\le \beta x  \]
for $x\in [\frac{1}{T},1]$. 

Let us fix a triple $(\io,i,j)$ and hence the measure $\nu_\io^{i,(j)}$. For simplicity, we write $\nu^*=\nu_\io^{i,(j)}$ and $f^*=f_\io^{i,(j)}$. When we deal with a fixed triple $(\io,i,j)$, we use this abbreviation; when we need to distinguish between them, we restore the indices $(\io,i,j)$. This will not cause any confusion. 

We apply Lemma \ref{structurelem} with parameters $u=\al,\kappa=\beta, f=f^*$. To verify the assumption for $f^*$, we just need to properly choose parameters $T^{-1}\ll \eta\ll \e\ll 1$ (see \eqref{choiceofpar}). Now we think of $\al,\beta$ as fixed constants. As a result of Lemma \ref{structurelem}, we obtain $k\le C(\eta)$ and $k$ disjoint intervals $\{[a_j,b_j]\}_{j=1}^k$ such that 
    \begin{enumerate}
        \item For $1\le j\le {k}$, $[a_j,b_j]$ is admissible and
        \begin{equation}\label{Frostmaninab}
            f^*(x)\ge f^*(a_j)+(1+\e)(x-a_j)-\eta(b_j-a_j)\ \ \ \ x\in[a_j,b_j]. 
        \end{equation} 
        \item $|[0,b_k]\setminus \bigcup_{j=1}^k [a_j,b_j]|\le \eta$.
        \item $f^*(b_{k})=(2\beta-1)b_{k}-\beta+1+O(\e)$.
    \end{enumerate}

\section{Multiscale good-bad decomposition}\label{sec6}

We follow the multiscale good-bad decomposition from \cite[Section 3.2]{liu2026lebesgue}.
Let \begin{equation}\label{defrhor}
\rho_j:=R_\io^{-a_j}, r_j=R_\io^{-b_j}
\end{equation} 
for $1\le j\le k$.
Then we have
\[ 1\ge \rho_1>r_1\ge \rho_2>r_2\dots\ge r_k>R_\io^{-1}. \]
Let us reinterpret the properties for $\{a_j,b_j\}$ in terms of $\{\rho_j,r_j\}$.
Since each $[a_j,b_j]$ is admissible, we obtain the key property 
\begin{equation}\label{increspeed}
    R_\io\rho_j\le (R_\io r_{j})^2.
\end{equation} 
By \eqref{f(1)}, \eqref{alphakt} and \eqref{verifybranching}, for $R_\io^{-1}\le r\le 1$ and each $Q_r\in \cD_r$ with $Q_r\cap \supp\nu^*\neq\emptyset$, we have the following $\al$-Katz-Tao condition:
\begin{equation}\label{alkt}
    |\supp\nu^*\cap Q_r|_{R_\io^{-1}}\sim \frac{\nu^*(Q_r)}{\nu^*(Q_{R_\io^{-1}})}\lesssim R_\io^{O(T^{-1})} R_\io^{f^*(1)-f^*(-\log_{R_\io} r)}\lesssim R_\io^{O(T^{-1})}(R_\io r)^{\al}.
\end{equation}

From \eqref{Frostmaninab}, we have the following $(1+\e)$-Frostman condition for $1\le j\le k$:
\begin{equation}\label{nustarfrost}
    \nu^*(Q_r)\lesssim (\frac{\rho_j}{r_j})^\eta (\frac{r}{\rho_j})^{1+\e} \nu^*(Q_{\rho_j}),
\end{equation}  
for all $r\in [r_j,\rho_j]$ and $Q_r\subset Q_{\rho_j}$.

The other two properties are
\begin{equation}\label{tinyrrho}
    \frac{1}{\rho_1}\frac{r_1}{\rho_2}\frac{r_2}{\rho_3}\dots\frac{r_{k-1}}{\rho_k}\le R_\io^\eta,
\end{equation}
and
\begin{equation}\label{rkcoveringnumber}
    |\supp \nu^*|_{r_k}\lesssim r_k^{1-2\beta} R_\io^{-\beta+1+O(\e)}.  
\end{equation} 

\bigskip

Now we are ready to perform the multiscale good-bad decomposition.
For each $1\le j\le k$, we use $\{\theta_j\}$ to denote a collection of $(r_j/\rho_{j})$-caps that form a partition of $S^1$.

For each cube $Q$ with $\nu^*(Q)>0$, define the normalized restriction
\[
\nu_Q^* := \frac{1}{\nu^*(Q)}\,\nu^*|_Q.
\]
This is the conditional measure associated with the already fixed measure $\nu^*$.

\begin{definition}[Bad caps]
Let $j\in \{1,\dots,k\}$. For each $(r_j/\rho_j)$-cap $\theta_j$, let $\T'_{j,\theta_j}$ be a family of grid $r_j\times\rho_j$-tubes covering $\R^2$ in direction $\theta_j$. For each dyadic $r_j$-cube $Q_j$, we say $\theta_j$ \textbf{is bad for} $Q_j$ if there exists $T\in\T'_{j,\theta_j}$ with $10T\cap Q_j\neq\emptyset$ such that
\begin{equation}\label{def-bad-Q-r-j}
\nu^*_{\wt Q_{j}}(T)>R_\io^{T^{-1/3}}(\rho_j/r_j)^{10\eta}\cdot (r_j/\rho_j),
\end{equation}
where $\wt Q_{j}\in\cD_{\rho_{j}}$ is the dyadic $\rho_{j}$-cube containing $Q_j$. For each $Q_j\in\cD_{r_j}$, denote
\begin{equation}
	\label{def-Theta-j-Q-j}
	\Theta_j(Q_j):=\{\theta_j: (r_j/\rho_j)\text{-caps, bad for } Q_j\}.
\end{equation} 
\end{definition}

\begin{definition}[Bad caps at all scales]
    Fix $R_\io^{-1}\le r\le 1$ and $Q_r\in\cD_r$. For each $1\le j\le k$ such that $r\le r_j$, we define
    \begin{equation}
	\label{def-Theta-j-Q-j'}\Theta_j(Q_{r}):=\Theta_j(Q_{j}),
\end{equation}
where $Q_j\in \cD_{r_j}$ is the $r_j$-cube containing $Q_r$. Define
\begin{equation}
    \Theta(Q_r):=\bigcup_{j:r_j\ge r}\Theta_j(Q_r),
\end{equation}
which is the set of bad caps for $Q_r$ at all scales.
\end{definition}

For a tube $T$, we use $\theta(T)\in S^1$ to denote the direction of $T$. 

\begin{definition}[Bad/good tubes]\label{badtubedef}
    Fix $R_\io^{-1}\le r\le 1$ and $Q_r\in\cD_r$. For each $1\le j\le k$ such that $r\le r_j$, we define
    \begin{equation}
	\label{def-T-i-Q-j}\T_{\io,Q_r,j}:=\{T\in\T_\iota: 2T\cap Q_r\neq\emptyset \text{ and } \exists\, \theta_j\in\Theta_j(Q_r) \text{ with } \theta(T)\in\theta_j \}.
\end{equation}
Define
\begin{equation}\label{def-T-i-Q}
    \T_{\iota, Q_r}:=\bigcup_{j: r_j\ge r}\T_{\iota, Q_r,j},
\end{equation}
which is the collection of tubes in $\T_\io$ whose directions lie in some cap in $\Theta(Q_r)$. These are bad tubes.
Also define the good tubes
\begin{equation}\label{defgoodtube}
    \T_{\io,Q_r,good}:=\T_{\io}\setminus \T_{\io,Q_r}.
\end{equation}
\end{definition}

Now we are ready to decompose $\mu$ into good and bad parts. Recall the definitions of the operators $M_0, M_T$ from \eqref{MT}. By \cite[Lemma 3.4]{guth2020falconer},
we have the following decomposition
\[ \mu=M_0\mu+\sum_{\io\ge 1}\sum_{T\in \T_\io}M_T \mu +\mu_\err,\] where $\|\mu_\err\|_{L^1}\lesssim \RD(R_0)\|\mu\|_{L^1}.$

We also denote $\mu_\io=\sum_{T\in\T_\io} M_T\mu$. Then we have
\begin{equation}\label{decomposemu} \mu=M_0\mu+\sum_{\io\ge 1} \mu_\io+\mu_\err. 
\end{equation}

We fix the $\io$ as we did at the end of the last section. 
For each dyadic $r$-cube $Q\in\cD_{r}$, define
\begin{equation}\label{def-good-i-Q}\mu_{good, \io, Q}:= \sum_{T\in \T_{\iota, Q,good}} M_T \mu.\end{equation}
For each $y\in \supp \nu^*$, define
\begin{equation}\label{def-mu-good-y}\mu_{good,\io, y}:=\mu_{good, \iota, Q(y)},\end{equation}
where $Q(y)$ denotes the dyadic $R_\io^{-1}$-cube containing $y$.

\medskip
For any $x\in \R^2$, define $d^x(y):=|x-y|$. For any measure $\mu$, define $d^x_*(\mu)$ as the pushforward measure.
For each $y\in\supp\nu^*$,
$$\begin{aligned}\|d_*^y(\mu_{good,\io, y})-d_*^y(\mu_\io)\|_{L^1(\R)}&=\|\sum_{T\in\T_{\iota, Q(y)}}d_*^y(M_T\mu)\|_{L^1(\R)}\\&\leq \sum_{T\in\T_{\iota, Q(y)}}\|d_*^y(M_T\mu)\|_{L^1(\R)}\\
(\textup{by \cite[Corollary 3.3]{guth2020falconer}})&  \lesssim \sum_{T\in \T_{\io,Q(y)}}\mu(2T)+\RD(R_\io)
   \end{aligned}$$
Since $\supp \mu$ and $\supp \nu$ are $\gtrsim 1$ apart, we see that the tubes in $\T_{\io,Q(y)}$, which are centered at $Q(y)$, are $O(1)$-overlapping on $\supp\mu$. Letting
$$Bad_{\iota}(y) := \bigcup_{T \in \T_{\iota, Q(y)}} 2T, $$ 
we have the inequality
 \begin{equation}
     \|d_*^y(\mu_{good,\io, y})-d_*^y(\mu_\io)\|_{L^1(\R)}\lesssim \mu ( Bad_{\iota}(y)) + \RD(R_\io).
 \end{equation}

Then, with
\begin{equation}\label{def-Bad-i}\begin{aligned}Bad_\iota:=&\{(x,y)\in \supp\mu\times \supp\nu:x\in Bad_\iota(y)\}\\=&\{(x,y)\in \supp\mu\times \supp\nu: \exists \,T\in  \T_{\iota, Q(y)},\, x\in 2T\},\end{aligned}\end{equation}
it follows that
   \begin{equation}\label{badpartest}
   \begin{split}\int \|d_*^y(\mu_{good,\io,y})-d_*^y(\mu_\io)\|_{L^1(\R)}d\nu^*(y)\lesssim &  \int \mu (Bad_\iota(y))d\nu^*(y)+ \RD(R_\io)\\= & \mu\times\nu^* (Bad_\iota)+ \RD(R_\io).
   \end{split}
   \end{equation}
Recall that $\nu^*=\nu_\io^{i,(j)}$ by our convention at the end of the last section.
   
Then we need the following geometric observation on pairs $(x,y)\in Bad_\iota$: by our definition of ``bad" in \eqref{def-bad-Q-r-j}-\eqref{def-T-i-Q} and \eqref{def-Bad-i}, there exist $1\leq j\leq k$ and $T\in\T_\iota$ such that $x,y\in 10T$ and $\theta(T)$ lies in a $(r_j/\rho_j)$-cap $\theta_j$ that is bad for $Q_j(y)\in\cD_{r_j}$ ($Q_j(y)$ is the $r_j$-cube containing $y$). Therefore, if we use
$$T_{r_j\times \rho_j}(y;x)$$
to denote the $r_j\times \rho_j$-tube centered at $y$ parallel to the line $l_{x,y}$ connecting $x,y$, then 
$$100T_{r_j\times \rho_j}(y;x)\supset T_{r_j\times \rho_j},$$
where $T_{r_j\times \rho_j}$ is a tube in direction $\theta_j$ from \eqref{def-bad-Q-r-j}, with $10T_{r_j\times \rho_j}\cap Q_j(y)\neq\emptyset$. Consequently,
$$\nu^*_{\wt Q_{j}}(100T_{r_j\times \rho_j}(y;x))>R_\io^{T^{-1/3}}(\rho_j/r_j)^{10\eta}\cdot (r_j/\rho_j),$$
and therefore the set $Bad_\iota$ is a subset of
\begin{equation}\label{badiota}
    \{(x,y)\in E_1\times E_2: \exists \,j,\,\nu^*_{\wt Q_{j}(y)}(100T_{r_j\times \rho_j}(y;x))>R_\io^{T^{-1/3}}(\rho_j/r_j)^{10\eta}\cdot (r_j/\rho_j)\}.
\end{equation}
Here, recall $E_1,E_2$ are as in Proposition \ref{reducedprop}.

We prove the following estimate for the bad part:
\begin{lemma}\label{badestimatelem}
Let $Bad_\io$ be defined as above. Then
    \begin{equation}\label{measure-Bad-i}\mu\times\nu^*(Bad_\iota)\lesssim  R_\io^{-\delta}\nu^*([0,1)^2).\end{equation}
\end{lemma}

To prove Lemma \ref{badestimatelem}, we just need to prove the following result.
\begin{proposition}\label{prop-good-bad}
	Let $\mu,\nu^*$ be the measures as above.  Recall that $\rho_j=R_\io^{-a_j},r_j=R_\io^{-b_j}$ from the beginning of this section. Then there exists a subset $Bad\subset [0,1)^4$ such that
	\begin{itemize}
		\item $\mu\times\nu^*(Bad)\lesssim  R_\io^{-\delta}\nu^*([0,1)^2)$,
		\item for all pairs of points $(x,y)\notin Bad$ and all $j\in \{1,\dots,k\}$, the $r_j\times \rho_j$ tube $T_{r_j\times \rho_j}(y;x)$ centered at $y$ parallel to the line $l_{x,y}$ connecting $x,y$ satisfies
		\begin{equation}\label{neglectRiota}
		    \nu^*_{\wt Q_j(y)}(100T_{r_j\times \rho_j}(y;x))< R_\io^{T^{-1/3}}(\rho_j/r_j)^{10\eta}\cdot (r_j/\rho_j),
		\end{equation}
		where the notation is the same as in \eqref{badiota}.
	\end{itemize}
	
\end{proposition}

From the property of $Bad$, we see that $(x,y)\notin Bad \Rightarrow (x,y)\notin Bad_\io$. Hence, $Bad_\io\subset Bad$, and this gives the proof of \eqref{measure-Bad-i}.

Before proving the proposition, we recall the radial projection theorem of Orponen \cite{orponen2018dimension}. We use the formulation as in \cite[Theorem 3.7]{guth2020falconer}. Denote $\pi^y(x):=\frac{y-x}{|y-x|}$ as the radial projection.
\begin{lemma}[Orponen, \cite{orponen2018dimension}]
     For every $\al>1$, there exists $p=p(\al)>1$ so that the following holds. Suppose $\mu$ and $\om$ are measures on the unit disk with disjoint supports and that for every ball $B(x,r)$, $\mu(B(x,r))\lesssim r^\al$ and $\om(B(x,r))\lesssim r^\al$. Then
     \begin{equation}\label{orponenprojection}
         \int \|\pi^y_* \mu\|_{L^p}^p d\om(y)\lesssim 1.
     \end{equation}
\end{lemma}

\begin{proof}[Proof of Proposition \ref{prop-good-bad}]
Fix $j\in\{1,\dots,k\}$. 

{\bf Step 1.} We first consider pairs $(x,y)\in\supp\mu\times\supp\nu^*$ with
$$\mu(T_{r_j/\rho_j\times 1}(x;y))>R_\io^{T^{-1/2}}\cdot r_j/\rho_j,$$
and denote the set of such pairs by $Bad_{j,\mu}$, which is an open subset of $[0,1)^4$. 

For each $y\in\supp\nu^*$, as $\dist(y, \supp\mu)\sim 1$, one can cover $\supp\mu$ by $\sim \rho_j/r_j$ many $r_j/\rho_j\times 1$ tubes containing $y$ of bounded overlapping. Denote this collection of tubes by $\T_{y, r_j/\rho_j\times 1}$. Then, for each $x$ with $(x,y)\in Bad_{j,\mu}$, there exists $T\in \T_{y, r_j/\rho_j\times 1}$ such that 
$$T_{r_j/\rho_j\times 1}(x;y)\subset 2T.$$
Because of bounded overlapping, there are $\lesssim R_\io^{-T^{-1/2}}\cdot \rho_j/r_j$ many $T\in \T_{y, r_j/\rho_j\times 1}$ with
$$\mu(2T)> R_\io^{T^{-1/2}}\cdot r_j/\rho_j.$$
Consequently, for each $y\in\supp\nu^*$,
$$\{x:(x,y)\in Bad_{j,\mu}\}$$
is contained in a union of $\lesssim R_\io^{-T^{-1/2}}\cdot \rho_j/r_j$ many $r_j/\rho_j\times 1$ tubes containing $y$ and hence has measure $\lesssim R_\io^{-T^{-1/2}}$. Therefore, if we denote $\pi^y(x):=\frac{y-x}{|y-x|}$ as the radial projection, then
$$\mu\times\nu^*(Bad_{j,\mu})=\int \mu(\{x:(x,y)\in Bad_{j,\mu}\})\,d\nu^*(y)= \int \left(\int \pi^y_*\mu\right) d\nu^*(y),$$
and for each $y\in\supp\nu^*$, the support of $\pi^y_*(\mu)$ is contained in $\lesssim R_\io^{-T^{-1/2}}\cdot \rho_j/r_j$ many $r_j/\rho_j$-arcs in $S^1$.

By \eqref{f(1)},\eqref{alphakt} and \eqref{verifybranching}, we see that
\[ \nu^*(B(x,r))\lesssim R_\io^{O(T^{-1})} r^\al \nu^*([0,1)^2). \]
Let $\om=R_\io^{-O(T^{-1})}\Big(\nu^*([0,1)^2)\Big)^{-1}\nu^*$, then $\om(B(x,r))\lesssim r^\al$.
Hence, by Orponen's radial projection estimate \eqref{orponenprojection}, there exists $p>1$ such that
\begin{equation}\label{badjmu}
    \mu\times\nu^*(Bad_{j,\mu})\lesssim R_\io^{O(T^{-1})}\nu^*([0,1)^2) R_\io^{-T^{-1/2}/p'} \cdot \int \|\pi^y_*\mu\|_{L^p}d\om(y)\lesssim R_\io^{-C T^{-1/2}}\nu^*([0,1)^2),
\end{equation}
where the constant $C$ in the exponent only depends on $\al$, and $R_\io^{O(T^{-1})}$ is absorbed into $R_\io^{-CT^{-1/2}}$ since $T\gg 1$.

\medskip
{\bf Step 2.} On $[0,1)^4\backslash Bad_{j,\mu}$, write 
$$\mu\times\nu^*=\sum_{Q\in\cD_{\rho_j}}\nu^*(Q)\cdot\mu\times\nu^*_Q.$$

For each $Q\in\cD_{\rho_j}$ and $y\in Q$, consider pairs $(x,y)\in[0,1)^2\times Q$ with
\begin{equation}\label{bad-r'-r-tube}\nu^*_Q(100T_{r_j\times \rho_j}(y;x))> R_\io^{T^{-1/3}}(\rho_j/r_j)^{10\eta}\cdot (r_j/\rho_j),\end{equation}
and denote the set of these pairs by $Bad_{j,Q}\subset (Bad_{j, \mu})^c$. 

Choose $\sim \rho_j/r_j$ many $r_j/\rho_j$-separated directions $\{e_i\}\subset S^1$. For each selected direction $e_i$, cover $Q$ by $\sim \rho_j/r_j$ many $r_j\times \rho_j$ tubes parallel to $e_i$ of bounded overlapping. Denote the collection of all selected $r_j\times \rho_j$ tubes by $\T_{Q, r_j\times \rho_j}$.

We find another characterization for $Bad_{j,Q}$.
Here comes the key observation in this step: for each $(x,y)\in Bad_{j,Q}$, there exists $T\in \T_{Q, r_j\times \rho_j}$ such that $T_{r_j\times \rho_j}(y;x)\subset 2T$. $T$ is associated to a $r_j/\rho_j\times 1$ tube  $T_{r_j/\rho_j\times 1}$ that has the same central line as $T$.  We denote their relationship by $T_{r_j/\rho_j\times 1}\sim T$. (We remark that different $T\in \T_{Q,r_j\times \rho_j}$ may be associated with the same  $T_{r_j/\rho_j\times 1}$.) 
We see that 
$$100T_{r_j\times \rho_j}(y;x)\subset 200T \text{ and } x\in T_{r_j/\rho_j\times 1}.$$

This observation implies that, if we call $T\in \T_{Q, r_j\times \rho_j}$ \textbf{bad} when $$\nu_Q^*(200T)> R_\io^{T^{-1/3}}(\rho_j/r_j)^{10\eta}\cdot (r_j/\rho_j),$$ then
$$Bad_{j,Q}\subset \bigcup_{\substack{T\in \T_{Q, r_j\times \rho_j}, \,bad\\T_{r_j/\rho_j\times 1}\sim T}}T_{r_j/\rho_j\times 1}\times T.$$
Since $(x,y)\notin Bad_{j, \mu}$, by our construction from Step 1, it follows that
$$\mu(T_{r_j/\rho_j\times 1})< R_\io^{T^{-1/2}}\cdot r_j/\rho_j.$$
Therefore,
$$\mu\times\nu^*_Q(Bad_{j,Q})\leq \sum_{\substack{T\in \T_{Q, r_j\times \rho_j}, \,bad\\T_{r_j/\rho_j\times 1}\sim T}}\mu(T_{r_j/\rho_j\times 1})\nu^*_Q(T)\leq R_\io^{T^{-1/2}}\cdot r_j/\rho_j \sum_{T\in \T_{Q, r_j\times \rho_j}, \,bad}\nu^*_Q(T)$$
$$=R_\io^{T^{-1/2}}\cdot r_j/\rho_j \sum_{\substack{e_i\in S^1 \\r_j/\rho_j\text{-separated}}} \sum_{\substack{T\in \T_{Q, r_j\times \rho_j}\\T\parallel e_i, \,bad}} \nu^*_Q(T).$$
As tubes in $\T_{Q, r_j\times \rho_j}$ have bounded overlap, by \eqref{bad-r'-r-tube}, for each $e_i$, the number of tubes $T$ in the sum is $\lesssim R_\io^{-T^{-1/3}}(\rho_j/r_j)^{10\eta} (\rho_j/r_j)$. So, by Cauchy-Schwarz,
$$\begin{aligned}&\left(\mu\times\nu^*_Q(Bad_{j,Q})\right)^2\\\lesssim &  (R_\io^{T^{-1/2}} r_j/\rho_j)^2\cdot \rho_j/r_j \cdot R_\io^{-T^{-1/3}}(\rho_j/r_j)^{-10\eta} (\rho_j/r_j)\sum_{\substack{e_i\in S^1 \\r_j/\rho_j\text{-separated}}} \sum_{\substack{T\in \T_{Q, r_j\times \rho_j}\\T\parallel e_i, \,bad}} \nu^*_Q(T)^2\\
\leq &R_\io^{2T^{-1/2}-T^{-1/3}}(\frac{\rho_j}{r_j})^{-10\eta} \sum_{T\in \T_{Q, r_j\times \rho_j}} \nu^*_Q(T)^2\\
= &R_\io^{2T^{-1/2}-T^{-1/3}}(\frac{\rho_j}{r_j})^{-10\eta}  \iint \sum_{T\in \T_{Q, r_j\times \rho_j}}\chi_{T}(y)\chi_{T}(y')\, d\nu^*_Q(y)\,d\nu^*_Q(y').
\end{aligned}$$

Fix $y'$ and consider $|y-y'|\sim 2^{-l}$ with $2^{-l}\in[r_j, \rho_j]$. Then the number of $T\in \T_{Q, r_j\times \rho_j}$ containing both $y,y'$ is $\lesssim 2^l\rho_j$. Also, since $\nu^*$ satisfies \eqref{nustarfrost}, it follows that
$$\nu^*_Q(B(y, 2^{-l}))\lesssim (\frac{\rho_j}{r_j})^\eta (\frac{2^{-l}}{\rho_j})^{1+\e}.$$
Together we have
$$\left(\mu\times\nu^*_Q(Bad_{j,Q})\right)^2 \lesssim R_\io^{-\de }(\frac{\rho_j}{r_j})^{-9\eta} \left( (r_j/\rho_j)^{1+\e}+\sum_{l:2^{-l}\in[r_j, \rho_j]} 2^l \rho_j\cdot(2^{l}\rho_j)^{-1-\e}\right).$$
Here, the first term in the bracket comes from $|y-y'|\le r_j$, while the second term comes from $|y-y'|\sim 2^{-l}$.
We see that
\begin{equation}\label{badjQ}
    \mu\times\nu^*_Q(Bad_{j,Q})\lesssim R_\io^{-\de }. 
\end{equation} 

\medskip
{\bf Step 3.} Finally take
$$Bad:=\bigcup_{j=1}^k Bad_{j,\mu}\cup\bigcup_{j=1}^k \bigcup_{Q\in\cD_{\rho_j}}Bad_{j,Q}.$$
From \eqref{badjmu} and \eqref{badjQ},
$$\mu\times\nu^*(Bad)\lesssim \sum_{j=1}^k\left(R_\io^{-CT^{-1/2}}\nu^*([0,1)^2)+\sum_{Q\in\cD_{\rho_j}}\nu^*(Q)R_\io^{-\de}\right)\lesssim R_\io^{-\delta }\nu^*([0,1)^2).$$
Here, we used that $k\le C(\eta)$ is bounded by a constant.
\end{proof}

\section{An estimate for the good part}\label{sec7}

We still work with the $\mu, \nu^*$ as we did in the last section. Recall that $\mu_{good, \io,y}$ is defined in \eqref{def-mu-good-y}. The purpose of this section is to estimate
\[ \int \|d_*^y(\mu_{good,\io,y})\|_{L^2(\R)}^2d\nu^*(y). \]
The main result is stated as follows.

\begin{proposition}\label{prop-our-L2-good}
Let $\mu,\nu^*$ be as in the last section. Then
\begin{equation}
\int \| d^y_*(\mu_{good,\io,y}) \|_{L^2(\R)}^2d\nu^*(y)<R_\io^{-\de}\nu^*([0,1)^2).
\end{equation}
\end{proposition}

By the observation in \cite{liu20192}, for each $y\in\supp\nu^*$,
$$ \| d^y_*(\mu_{good,\io,y}) \|_{L^2(\R)}^2 = \int_0^\infty | \mu_{good,\io,y} * \sigma_t (y)|^2 t^2 dt,$$
where $\sigma_t=d\si_t$ denotes the normalized arc-length measure on $S^1_t$, the circle of radius $t$. This is
$$\lesssim \int_0^\infty | \mu_{good,\io,y} * \sigma_t (y)|^2 t dt $$
because $\mu_{good,\io,y}$ is essentially supported near $\supp\mu$, away from $y\in\supp\nu^*$ (see Lemma 5.2 in \cite{guth2020falconer}). Then, by the $L^2$-identity proved in \cite{liu20192}, for every $y \in \R^2$, 
$$\int_0^\infty | \mu_{good, \io,y} * \sigma_t (y)|^2 t dt=\int_0^\infty | \mu_{good, \io,y} * \wh{\sigma_r} (y)|^2 r dr.$$
By the definition of $\mu_{good,\io, Q(y)}$, $\mu_{good, \io,y}$ in \eqref{def-good-i-Q}, \eqref{def-mu-good-y}, we see that the Fourier support of $\mu_{good, \io,y}$ essentially lies in $|\xi|\sim R_\io$. We hence have
\begin{equation}\label{fixR}
    \int_0^\infty | \mu_{good,\io, Q(y)} * \wh{\sigma_r} (y)|^2 r dr\lesssim \int_{r\sim R_\io} | \mu_{good,\io, Q(y)} * \wh{\sigma_r} (y)|^2 rdr+\RD(R_\io).
\end{equation}

Notice that the dependence on $y$ in $\mu_{good, \io,y}$ has not yet made any difference because so far everything is pointwise. 

Now we fix $R\sim R_\io$ and consider the integral with respect to $d\nu^*(y)$:
\begin{equation}\label{L2-good-formulation}\int | \mu_{good,\io, Q(y)} * \wh{\sigma_R} (y)|^2 d\nu^*(y)= \int | (\widehat{\mu_{good,\io, Q(y)}} d\sigma_R)^\vee (y)|^2 d\nu^*(y).
\end{equation}

To deal with the dependence on $Q(y)$, a careful analysis on bad caps associated with different $y$ is required. 

Recall \eqref{defrhor} that $\rho_j=R_\io^{-a_j}, r_j=R_\io^{-b_j}$, and 
\[R_\io^{-1}< r_k<\rho_k<\dots< r_1<\rho_1\le 1.\]
For convenience, we introduce the following notation.
\begin{definition}\label{defF}
    Define 
    \[F:=(\wh\mu_\io d\si_R)^\vee= (\sum_{T\in\T_\io}\wh{M_T \mu}\ d\si_R)^\vee.\]
    For each $T\in\T_\io$, define
    \[ F_T:=(\wh{M_T \mu}\ d\si_R)^\vee. \]
    For each $R_\io^{-1/2}$-cap $\theta$, define
    \[ F_\theta:=(\sum_{T\in\T_{\io,\theta}}\wh{M_T\mu}\  d\si_R)^\vee=\sum_{T\in\T_{\io,\theta}}F_T. \]
    For $r\ge R_\io^{-1/2}$ and an $r$-cap $\tau$, define
    \[ F_\tau:=\sum_{\theta\subset \tau}F_\theta. \]
\end{definition}

\begin{lemma}\label{lemrapdec}
    For each $T\in \T_\io$, $F_T 1_{B(0,1)}$ is essentially supported in $T$.
\end{lemma}
\begin{proof}
Let $S^1_R$ be the circle of radius $R$.
    Note that $\wh{M_T \mu}$ is essentially supported in a $R_\io^{1/2}\times R_\io$-rectangle. Therefore $\wh{M_T\mu}\ d\si_R$ is essentially supported in an arc of length $R_\io^{1/2}$ in $S_R^1$, denoted by $J$. We have
    \[ F_T\approx (\wh{M_T\mu}\  1_J d\si_R)^\vee=M_T\mu*(1_Jd\si_R)^\vee. \]
    Since inside $B(0,1)$, $M_T\mu$ is essentially supported in $T$ and $(1_J d\si_R)^\vee$ is essentially supported in a $R_\io^{-1/2}\times 1$-rectangle $T_0$, which is the translation of $T$ to the origin, we see that $F_T$ is essentially supported in $T$.
\end{proof}

We also define the good functions which may be compared with Definition \ref{badtubedef}.
\begin{definition}\label{defFQ}
Fix $R_\io^{-1}\le r\le 1$ and $Q_r\in\cD_r$. Define
\begin{equation}
    F_{Q_r}:=\sum_{T\in\T_{\io,Q_r,good}}F_T.
\end{equation}
\end{definition}
\begin{remark}
    {\rm
    We only care about the value of $F_{Q_r}$ inside $Q_r$.
    }
\end{remark}
From the definition, we see that for each $Q\in\cD_{R_\io^{-1}}$ and any $y\in Q$,
\[ (\wh{\mu_{good,\io,y}}d\si_R)^\vee=F_Q.  \]

Now we can rewrite the integrand in \eqref{L2-good-formulation} as
\begin{equation}\label{sumFQ}
    \int |(\wh{\mu_{good,\io,y}}d\si_R)^\vee(y)|^2 d\nu^*(y)= \sum_{Q\in\cD_{R_\io^{-1}}}\int_Q |F_Q(y)|^2 d\nu^*(y).
\end{equation}
Since $F_Q$ is the sum of some wave packets from $\T_\io$, we see that $\wh F_Q$ is essentially supported in $B(0,R_\io)$. By the locally constant property, $F_Q$ is essentially constant in each $Q\in\cD_{R_\io^{-1}}$. Now, for each $r$, we denote
\[ \cD_r(\nu^*):=\{Q_r\in\cD_r: Q_r\cap \supp\nu^*\neq\emptyset\}, \]
\[ X_r:=\bigcup_{Q\in\cD_r(\nu^*)}Q. \]
Therefore, up to a rapidly decaying error, we may bound \eqref{sumFQ} by
\begin{equation}\label{averaged}
    \sum_{Q\in\cD_{R_\io^{-1}}(\nu^*)} \frac{\nu^*(Q)}{|Q|} \int_Q |F_Q(y)|^2dy. 
\end{equation} 
Since $\nu^*(Q)$ are comparable for all $Q\in\cD_{R_\io^{-1}}(\nu^*)$, we may bound the above quantity by
\begin{equation}\label{toestbyinflation}
    \frac{\nu^*([0,1)^2)}{|X_{R_\io^{-1}}|}\sum_{Q\in \cD_{R_\io^{-1}}} \int_Q |F_Q|^2. 
\end{equation} 

Now, we estimate \eqref{toestbyinflation} by a ball inflation process. It consists of the following three lemmas.

\begin{lemma}\label{ballinflation1}
    For $j\in \{1,\dots,k\}$, we have
    \[ \frac{1}{|X_{\rho_j}|}\sum_{\wt Q\in \cD_{\rho_j}(\nu^*)}\int_{\wt Q}|F_{\wt Q}|^2\lesssim (\frac{r_{j-1}}{\rho_j})^2  \frac{1}{|X_{r_{j-1}}|}\sum_{Q\in \cD_{r_{j-1}}(\nu^*)}\int_Q|F_Q|^2. \]
    Here, $r_0=1$.
\end{lemma}
\begin{proof}
This is a trivial estimate.
We note that for $\wt Q\in\cD_{\rho_j}, Q\in \cD_{r_{j-1}}$ with $\wt Q\subset Q$, then $F_{\wt Q}=F_Q$ on $\wt Q$ up to rapidly decaying tail. Noting that the integration domain on the right-hand side contains the integration domain on the left-hand side, hence the factor we lose is $\le \frac{|X_{r_{j-1}}|}{|X_{\rho_j}|}=\frac{|Q|}{|\wt Q|}=(\frac{r_{j-1}}{\rho_j})^2$.
\end{proof}

\begin{lemma}\label{ballinflation2}
    For $j\in \{1,\dots,k\}$, we have
    \[\frac{1}{|X_{r_j}|} \sum_{Q\in \cD_{r_j}(\nu^*)}\int_{Q}|F_{Q}|^2\lesssim R_\io^{O(T^{-1/3})}(\frac{\rho_j}{r_j})^{O(\eta)}  \frac{1}{|X_{\rho_j}|} \sum_{\wt Q\in \cD_{\rho_{j}}(\nu^*)}\int_{\wt Q}|F_{\wt Q}|^2. \]
\end{lemma}
\begin{proof}
    Fix a $\wt Q\in\cD_{\rho_j}(\nu^*)$. Noting that $|X_{r_j}|/|X_{\rho_j}|\sim |X_{r_j}\cap \wt Q|/|\wt Q|$, it suffices to prove
    \[ \frac{1}{|X_{r_j}\cap\wt Q|}\sum_{Q\in \cD_{r_j}(\nu^*), Q\subset \wt Q}\int_Q|F_Q|^2\lesssim R_\io^{O(T^{-1/3})}(\frac{\rho_j}{r_j})^{O(\eta)}\frac{1}{|\wt Q|}\int_{\wt Q}|F_{\wt Q}|^2. \]

Let $\{\tau\}$ be the same partition into $r_j/\rho_j$-caps used to define $\Theta_j$ in \eqref{def-Theta-j-Q-j}. Let 
\begin{equation}\label{defFQtau}
F_{Q,\tau}:=\sum_{T\in\T_{\io,Q,good},\theta(T)\in\tau}F_T. 
\end{equation}
In particular, $F_{Q,S^1}=F_Q$.
(Recall the definition of $\T_{\io,Q,good}$ in \eqref{defgoodtube}.)

Note that $\wh F_{Q,\tau}$ is supported in the $Rr_j/\rho_j$-cap $R\tau=\{R\xi: \xi\in\tau\}$ in $S^1_R$ (where $S^1_R$ is the circle of radius $R$), and the convex hull of $R\tau$ is a rectangle of dimensions $Rr_j/\rho_j\times R(r_j/\rho_j)^2$. Noting that $\{R\tau\}$ are $Rr_j/\rho_j$-separated and $R\sim R_\io$, then since \eqref{increspeed} implies that $Rr_j/\rho_j\ge r_j^{-1}$, we use local $L^2$-orthogonality to deduce that
\[ \int_Q |F_Q|^2\lesssim \int_Q \sum_{|\tau|=r_j/\rho_j}|F_{Q,\tau}|^2,  \]
up to rapidly decaying tails.

For each $\tau$, let $\T'_\tau$ consist of the tubes in the fixed family $\T'_{j,\tau}$ that meet $\wt Q$.
We have
\begin{equation}\label{51}
\begin{split}
    \sum_{Q\in \cD_{r_j}(\nu^*), Q\subset \wt Q}\int_Q|F_Q|^2&\lesssim \sum_{Q\in \cD_{r_j}(\nu^*), Q\subset \wt Q} \sum_{|\tau|=r_j/\rho_j} \sum_{T'\in\T'_\tau}\int_{T'}|F_{Q,\tau}|^2 1_Q\\
    &=\sum_{|\tau|=r_j/\rho_j} \sum_{T'\in\T'_\tau}\int_{T'}\sum_{\substack{Q\in \cD_{r_j}(\nu^*),\ Q\subset\wt Q\\Q\cap T'\neq\emptyset}}|F_{Q,\tau}|^2 1_Q.
    \end{split}
\end{equation}

Fix a $\tau$ and $T'\in\T'_\tau$. For $Q\in\cD_{r_j}(\nu^*)$ with $Q\subset\wt Q$ and $Q\cap T'\neq \emptyset$,
we explore the relationship between $F_{Q,\tau}$ and $F_{\wt Q,\tau}$.  Recalling Definition \ref{badtubedef}, we see that
\begin{equation}\label{TQTQtilde}
    \bigg(\T_{\io,\wt Q}\cap \{T\in\T_\io: 2T\cap Q\neq\emptyset\}\bigg) \bigcup \T_{\io,Q,j} = \T_{\io, Q}.  
\end{equation} 
We introduce a convenient notation. For a collection of tubes $\T$ and a cap $\tau\subset S^1$, we denote
\[ \T[\tau]:=\{T\in\T:\theta(T)\in\tau\}. \]
In other words, $\T[\tau]$ is the subcollection of tubes whose directions lie in $\tau$.

Restricting \eqref{TQTQtilde} to direction $\tau$, we have
\begin{equation}\label{TQtau}
    \bigg( \T_{\io,\wt Q}\cap \{T\in\T_\io[\tau]: 2T\cap Q\neq\emptyset\}\bigg) \bigcup  \T_{\io,Q,j}[\tau] = \T_{\io, Q}[\tau].  
\end{equation}
We discuss two cases.

Case 1: $\nu_{\wt Q}^*(T')>R_\io^{T^{-1/3}}(\frac{\rho_j}{r_j})^{10\eta}(r_j/\rho_j)$. Fix any $Q\in\cD_{r_j}(\nu^*)$ with $Q \cap T'\neq\emptyset$. By \eqref{def-bad-Q-r-j}, we see that the $r_j/\rho_j$-cap $\tau$ (corresponding to the direction of $T'$) is \textbf{bad for} $Q$. Hence $ \T_{\io,Q}[\tau]\supset \T_{\io,Q,j}[\tau]$ contains all tubes in $\T_\io[\tau]$ whose $2$-dilations intersect $Q$. This implies $\{T\in\T_{\io,Q,good}[\tau]:2T\cap Q\neq\emptyset\}=\emptyset$. By \eqref{defFQtau} and Lemma \ref{lemrapdec}, $|F_{Q,\tau}|1_Q=\RD(R_\io)$.

Case 2: $\nu_{\wt Q}^*(T')\le R_\io^{T^{-1/3}}(\frac{\rho_j}{r_j})^{10\eta}(r_j/\rho_j)$. For each $Q$ meeting $T'$, if $\tau\in\Theta_j(Q)$, then $|F_{Q,\tau}|1_Q=\RD(R_\io)$ as above. Otherwise, $\T_{\io,Q,j}[\tau]=\emptyset$, and \eqref{TQtau} gives
\[ \T_{\io,\wt Q}\cap \{T\in\T_\io[\tau]: 2T\cap Q\neq\emptyset\} = \T_{\io, Q}[\tau].  \]
It further implies
\[ |F_{Q,\tau}|1_Q\lesssim |F_{\wt Q,\tau}|+\RD(R_\io). \]

Call $T'\in\T'_\tau$ retained if it meets a cube $Q\in\cD_{r_j}(\nu^*)$ with $Q\subset\wt Q$ and $\tau\notin\Theta_j(Q)$, and put $Y:=X_{r_j}\cap\wt Q$. By the two cases above, \eqref{51} is bounded by (up to negligible error)
\begin{equation}\label{52}
\begin{split}
    &\sum_{|\tau|=r_j/\rho_j} \sum_{\substack{T'\in\T'_\tau\\T'\textup{ retained}}}\int_{T'}\sum_{\substack{Q\in\cD_{r_j}(\nu^*),\ Q\subset\wt Q\\Q\cap T'\neq\emptyset,\ \tau\notin\Theta_j(Q)}}|F_{\wt Q,\tau}|^2 1_Q\\
    \le&\sum_{|\tau|=r_j/\rho_j} \sum_{\substack{T'\in\T'_\tau\\T'\textup{ retained}}}\int_{T'\cap Y}|F_{\wt Q,\tau}|^2.
    \end{split}
\end{equation}

Note that the integration domain $T'\cap Y$ is morally a collection of $r_j$-cubes. We  partition $\tau=\bigsqcup \si$ of $\tau$ into $(R\rho_j)^{-1/2}$-caps $\{\si\}$. This is allowable since $(R\rho_j)^{-1/2}\le r_j/\rho_j$ by \eqref{increspeed}. We obtain
\[ F_{\wt Q,\tau}=\sum_{\si\subset\tau} F_{\wt Q,\si}. \]
By the local $L^2$-orthogonality, we have
\begin{equation}\label{T'L2}
    \int_{T'\cap Y}|F_{\wt Q,\tau}|^2\lesssim \int_{T'\cap Y}\sum_{\si\subset\tau}|F_{\wt Q,\si}|^2. 
\end{equation}
Similar to the discussion in the paragraph below \eqref{defFQtau}, we have that $\supp \wh F_{\wt Q,\si}$ is essentially contained in a $R(R\rho_j)^{-1/2}\times R(R\rho_j)^{-1}$-rectangle. By locally constant property, $|F_{\wt Q,\si}|$ is essentially constant on rectangles of dimensions $R^{-1/2}\rho_j^{1/2}\times \rho_j$ whose long direction lies in $\si$. Noting that $R^{-1/2}\rho_j^{1/2}\le r_j$ by \eqref{increspeed}, we can tile $T'$ by $R^{-1/2}\rho_j^{1/2}\times \rho_j$-tubes $T_\si$ whose direction is $\si$, denoted by $T'=\bigcup T_\si$. We therefore have
\begin{equation}\label{Fsi}
    |F_{\wt Q,\si}|^2\sim \sum_{T_\si} c_{T_\si} 1_{T_\si} 
\end{equation} 
up to rapid decaying error inside $T'$.

If $T'$ is retained, choose such a cube $Q_0$. Cover $4T'$ by $O(1)$ neighboring grid tubes $S$. Since $T'\subset10S$, each $10S$ meets $Q_0$, so $\tau\notin\Theta_j(Q_0)$ implies
\[
\nu^*_{\wt Q}(4T')\lesssim R_\io^{T^{-1/3}}(\rho_j/r_j)^{10\eta}(r_j/\rho_j).
\]
Every $r_j$-cube meeting $T'$ lies in $4T'$. Thus, by \eqref{verifybranching},
\[
\begin{aligned}
N(T')&:=\#\{Q\in\cD_{r_j}(\nu^*):Q\subset\wt Q,\ Q\cap T'\neq\emptyset\}\\
&\lesssim R_\io^{O(T^{-1/3})}(\rho_j/r_j)^{10\eta}\frac{r_j}{\rho_j}\frac{|Y|}{r_j^2}.
\end{aligned}
\]
Writing $w=R^{-1/2}\rho_j^{1/2}$, we have $|T_\si\cap Q|\lesssim wr_j$ for each $r_j$-cube $Q$. Consequently,
\[
|T_\si\cap T'\cap Y|\lesssim wr_jN(T')
\lesssim R_\io^{O(T^{-1/3})}(\rho_j/r_j)^{10\eta}\frac{|T_\si|}{|\wt Q|}|Y|.
\]
Combining with \eqref{Fsi} and plugging into \eqref{T'L2}, we see that
\[  \int_{T'\cap Y}|F_{\wt Q,\tau}|^2\lesssim R_\io^{O(T^{-1/3})}(\frac{\rho_j}{r_j})^{10\eta}\frac{|\wt Q\cap X_{r_j}|}{|\wt Q|}\int_{T'}\sum_{\si\subset\tau}|F_{\wt Q,\si}|^2. \]
This implies 
\begin{align*}
    \eqref{52}&\lesssim R_\io^{O(T^{-1/3})}(\frac{\rho_j}{r_j})^{10\eta}\frac{|\wt Q\cap X_{r_j}|}{|\wt Q|}\sum_{|\tau|=r_j/\rho_j}\int_{\wt Q}|F_{\wt Q,\tau}|^2\\
    &\lesssim R_\io^{O(T^{-1/3})}(\frac{\rho_j}{r_j})^{10\eta}\frac{|\wt Q\cap X_{r_j}|}{|\wt Q|}\int_{\wt Q}|F_{\wt Q}|^2. 
\end{align*} 
In the last inequality, we used that $F_{\wt Q,\tau}$ are orthogonal on $\wt Q$. We have shown that they are orthogonal on $Q$, so they are also orthogonal on bigger cube $\wt Q$.

This finishes the proof.
\end{proof}

\begin{lemma}\label{ballinflation3}
    Recalling \eqref{alkt}, assume 
    \begin{equation}\label{ktnu}
        |\supp\nu^*\cap Q_r|_{R_\io^{-1}}\lesssim R_\io^{O(T^{-1})}(R_\io r)^\al 
    \end{equation} 
    for all $R_\io^{-1}\le r\le r_k$ and $Q_r\in\cD_r(\nu^*)$. Then
    \begin{equation}
        \sum_{Q\in \cD_{R_\io^{-1}}(\nu^*)}\int_Q |F_Q|^2\lesssim R^{O(T^{-1})}(\frac{|\supp\nu^*|_{R_\io^{-1}}}{|\supp\nu^*|_{r_k}})^{\frac{\al}{\al+1}}\frac{1}{R_\io r_k}\sum_{\wt Q\in \cD_{r_k}(\nu^*)}\int_{\wt Q} |F_{\wt Q}|^2.
    \end{equation}
\end{lemma}
\begin{proof}
    We fix a $\wt Q\in \cD_{r_k}(\nu^*)$. It suffices to prove
    \begin{equation}
        \sum_{Q\in \cD_{R_\io^{-1}}(\nu^*),Q\subset \wt Q}\int_Q |F_Q|^2\lesssim R^{O(T^{-1})}(\frac{|\supp\nu^*|_{R_\io^{-1}}}{|\supp\nu^*|_{r_k}})^{\frac{\al}{\al+1}}\frac{1}{R_\io r_k}\int_{\wt Q} |F_{\wt Q}|^2.
    \end{equation}
Recalling \eqref{def-T-i-Q} and noting that there is no $r_j$ in $[R_\io^{-1},r_k)$, we have
\[ \T_{\io,Q}=\T_{\io,\wt Q}\cap \{T:2T\cap Q\neq \emptyset\}. \]
This implies that
\[ |F_Q| 1_Q\lesssim |F_{\wt Q}| 1_{Q} \]
up to rapid decaying tails.
Therefore, it suffices to prove
\begin{equation}\label{rescaling}
    \int_{{X_{R_\io^{-1}}}\cap \wt Q} |F_{\wt Q}|^2\lesssim R^{O(T^{-1})}(\frac{|\supp\nu^*|_{R_\io^{-1}}}{|\supp\nu^*|_{r_k}})^{\frac{\al}{\al+1}}\frac{1}{R_\io r_k}\int_{\wt Q} |F_{\wt Q}|^2.
\end{equation}

Next, we rescale.
In the physical space, we apply the $R$-dilation followed by a translation, so that $\wt Q$ becomes $\wt B=Q_{Rr_k}$ and $\{Q\in\cD_{R^{-1}}(\nu^*):Q\subset \wt Q\}$ becomes a set of unit balls $\{B\}$ lying in $\wt B$. We denote
\[ G=F_{\wt Q}(R^{-1}x). \]
Then $G$ has Fourier transform supported in $S^1$. Under the rescaling, \eqref{rescaling} becomes
\begin{equation}
    \int_{\bigcup B}|G|^2\lesssim R^{O(T^{-1})}(\frac{|\supp\nu^*|_{R_\io^{-1}}}{|\supp\nu^*|_{r_k}})^{\frac{\al}{\al+1}}\frac{1}{R_\io r_k}\int_{\wt B} |G|^2.
\end{equation}
Since $\frac{|\supp\nu^*|_{R_\io^{-1}}}{|\supp\nu^*|_{r_k}}=\#\{Q\in\cD_{R_\io^{-1}}(\nu^*):Q\subset \wt Q\}=\#\{B\}\sim |\bigcup B|$, the above inequality is equivalent to
\begin{equation}\label{desiredL2}
    \int_{\bigcup B}|G|^2\lesssim R^{O(T^{-1})}|\bigcup B|^{\frac{\al}{\al+1}}\frac{1}{R_\io r_k}\int_{\wt B} |G|^2.
\end{equation}
This is a weighted $L^2$ estimate.
We are going to apply Lemma \ref{lemwtedL2} whose proof is postponed to Section \ref{secweightedL2}. By \eqref{ktnu}, we may partition $\bigcup B$ into $\lesssim R_\io^{O(T^{-1})}$ subsets, so that without loss of generality we may assume
\[ |\bigcup B \cap Q_r|_{1}\lesssim {r^\al} \]
for all $1\le r\le R_\io r_k$.

This means that after the $(R_\io r_k)^{-1}$-dilation, $\bigcup B$ becomes a Katz-Tao $((R_\io r_k)^{-1},\al)$-set in the sense of Definition \ref{defclassicalKT}.
Applying \eqref{weightedL2Ef} in Lemma \ref{lemwtedL2} with $X=\bigcup B, R=R_\io r_k, s=\al, \e=T^{-1}, Ef=G$, we obtain the desired estimate \eqref{desiredL2}.

This finishes the proof.
\end{proof}

\bigskip

We finish the proof of Proposition \ref{prop-our-L2-good} by combining Lemma \ref{ballinflation1}, \ref{ballinflation2}, \ref{ballinflation3} together.

\begin{proof}[Proof of Proposition \ref{prop-our-L2-good}]
By \eqref{fixR}, it suffices to show that
\[ \int_{R\sim R_\io} RdR\int|\mu_{good,\io,y}*\wh\si_R(y)|^2  d\nu^*(y)\lesssim R_\io^{-\de^4}\nu^*([0,1)^2). \]
By \eqref{sumFQ}, \eqref{averaged}, \eqref{toestbyinflation}, it suffices to show
\begin{equation}\label{toshowthis}
    \int_{R\sim R_\io} RdR \bigg(\frac{1}{|X_{R_\io^{-1}}|}\sum_{Q\in\cD_{R_\io^{-1}}(\nu^*)}\int|F_Q(y)|^2 dy\bigg) \lesssim R_\io^{-O(T^{-1})}. 
\end{equation} 
We remark that $F_Q$ depends on $R$.

Combining Lemma \ref{ballinflation1}, \ref{ballinflation2}, \ref{ballinflation3}, we see the left-hand side of \eqref{toshowthis} is bounded by
\[ (\frac{1}{\rho_1}\frac{r_1}{\rho_2}\dots\frac{r_{k-1}}{\rho_k})^2 R_\io^{kO(T^{-1/3})}(\prod_{j=1}^k\frac{\rho_j}{r_j})^{O(\eta)}\frac{|X_{r_k}|}{|X_{R_\io^{-1}}|}(\frac{|\supp\nu^*|_{R_\io^{-1}}}{|\supp\nu^*|_{r_k}})^{\frac{\al}{\al+1}}\frac{1}{R_\io r_k}\int_{R\sim R_\io} RdR \bigg(\int |F(y)|^2dy \bigg). \]
Here, $F=(\wh \mu_\io d\si_R)^\vee$ is given in Definition \ref{defF}.

By \eqref{tinyrrho}, $(\frac{1}{\rho_1}\frac{r_1}{\rho_2}\dots\frac{r_{k-1}}{\rho_k})^2\lesssim R_\io^{O(\eta)}$ is negligible. Since $k\le C(\eta)$, we can choose $T^{-1}$ sufficiently small so that $kO(T^{-1/3})\le \e$, and hence $R_\io^{kO(T^{-1/3})}\le R_\io^\e$.

Next, we estimate 
\[\int_{R\sim R_\io} RdR \bigg(\int |F(y)|^2dy \bigg)=\int_{R\sim R_\io} \int |(\wh \mu_\io d\si_R)^\vee(y)|^2RdydR.\]
Note that $d\si_R\sim \frac{1}{R}1_{S^1_R}$ is the normalized measure. The above integral is
\[ \sim \int_{R\sim R_\io}\int_{S^1_R} |\wh \mu_\io(\om)|^2 d\om R^{-1}dR\sim R_\io^{-1}\int_{|\xi|\sim R_\io} |\wh \mu_\io(\xi)|^2.   \]
Recall that the $s$-dimensional energy of the measure $\mu$ (in $\R^2$) is defined by
\[ I_s(\mu):=\int\int |x-y|^{-s}d\mu(x)d\mu(y)=c_s \int|\wh \mu(\xi)|^2|\xi|^{-2+s}. \]
Since $\mu$ satisfies the ball condition \eqref{frostcondition}, we may choose $\al'<\al$ slightly smaller, so that \eqref{cond1} and \eqref{cond2} hold for $(\al',\beta)$, and $I_{\al'}(\mu)<\infty$. For simplicity, we relabel $\al'=\al$. As a result,
\[ R_\io^{-1}\int_{|\xi|\sim R_\io} |\wh \mu_\io(\xi)|^2\lesssim R_\io^{1-\al}. \]

Combining the above discussion, the left-hand side of \eqref{toshowthis} is bounded by
\begin{equation}\label{finalterm}
    R_\io^{O(\e)}\frac{|X_{r_k}|}{|X_{R_\io^{-1}}|}(\frac{|\supp\nu^*|_{R_\io^{-1}}}{|\supp\nu^*|_{r_k}})^{\frac{\al}{\al+1}}\frac{1}{R_\io r_k} R_\io^{1-\al}. 
\end{equation} 
Now, we do the same calculation as we did to estimate \eqref{finalestimate}. We remark that the $u, \kappa$ in \eqref{finalestimate} become $\al,\beta$ in \eqref{finalterm}.

We see that \eqref{finalterm} is bounded by
\[ R_\io^{O(\e)} R_\io^{1-\al} R_\io^{\max\{\frac{\beta-\al}{2\beta-1},\frac{\beta-\al}{(\al+1)(2\beta-\al-1)}\}}. \]
Since $\al,\beta$ satisfy \eqref{cond1} and \eqref{cond2}, when $\e$ is sufficiently small, the above quantity is bounded by $R_\io^{-\de}$.

This finishes the proof of Proposition \ref{prop-our-L2-good}.
\end{proof}

\section{Proof of Proposition \ref{reducedprop}} \label{sec8}  
We prove Proposition \ref{reducedprop} in this section. This will imply Theorem \ref{maintheorem}.
Recall that at the end of Section \ref{sec5}, we denoted $\nu_\io^{i,(j)}$ by $\nu^*$ for convenience. Now we will recover the omitted indices. By \eqref{decomposemu}, we have

\begin{equation}\label{finalest}
    \begin{split}
        1&=\int\bigg( \int_{\De_y(E_1)}d_*^y(\mu) \bigg)d\nu (y)\\
&=\int\bigg( \int_{\De_y(E_1)}d_*^y(M_0\mu) \bigg)d\nu (y)+\sum_{\io\ge 1}\int\bigg( \int_{\De_y(E_1)}d_*^y(\mu_\io) \bigg)d\nu (y)+\RD(R_0)\\
&\le \bigg|\int\bigg( \int_{\De_y(E_1)}d_*^y(M_0\mu) \bigg)d\nu (y)\bigg|+\sum_{\io\ge 1}\bigg|\int\bigg( \int_{\De_y(E_1)}d_*^y(\mu_\io) \bigg)d\nu (y)\bigg|+\RD(R_0). 
    \end{split}
\end{equation}

We study the second term on the right-hand side.
Fix $\io\ge 1$. Recall that
\[ \nu=\sum_{i=1}^{N_\io}\sum_j\nu_\io^{i,(j)}+\nu_\io^{N_\io+1}. \]
Here, $\nu_\io^{i,(j)}$ were abbreviated as $\nu^*$ in the last few sections, and $\nu_\io^{N_\io+1}([0,1)^2)\le R_\io^{-\de^3}$ (see the paragraph below \eqref{MT}). 
We have
\begin{align*}
    \bigg|\int\bigg( \int_{\De_y(E_1)}d_*^y(\mu_\io) \bigg)d\nu (y)\bigg| &\le\bigg|\int\bigg( \int_{\De_y(E_1)}d_*^y(\mu_\io)-d_*^y(\mu_{good,\io,y}) \bigg)d\sum_{i=1}^{N_\io}\sum_j\nu_\io^{i,(j)}(y)\bigg| \\
    &+\bigg|\int\bigg( \int_{\De_y(E_1)}d_*^y(\mu_{good,\io,y}) \bigg)d\sum_{i=1}^{N_\io}\sum_j\nu_\io^{i,(j)}(y)\bigg|\\
    &+\bigg|\int\bigg( \int_{\De_y(E_1)}d_*^y(\mu_\io) \bigg)d\nu_\io^{N_\io+1} (y)\bigg|\\
    &=:I+II+III.
\end{align*} 

\begin{align*}
    I&\le \sum_{i=1}^{N_\io}\sum_j\bigg|\int\bigg( \int_{\De_y(E_1)}d_*^y(\mu_\io)-d_*^y(\mu_{good,\io,y}) \bigg)d\nu_\io^{i,(j)} (y)\bigg|\\
    &\le \sum_{i=1}^{N_\io}\sum_j \int \|d_*^y(\mu_{good,\io,y})-d_*^y(\mu_\io))\|_{L^1(\R)}d\nu_\io^{i,(j)}(y)\\
    (\textup{by~}\eqref{badpartest},\eqref{measure-Bad-i})&\lesssim \sum_{i=1}^{N_\io}\sum_j R_\io^{-\de}\nu_\io^{i,(j)}([0,1)^2)\\
    &\le R_\io^{-\de}.
\end{align*}

\begin{align*}
    II&\le \sum_{i=1}^{N_\io}\sum_j\bigg|\int\bigg( \int_{\De_y(E_1)}d_*^y(\mu_{good,\io,y}) \bigg)d\nu_\io^{i,(j)} (y)\bigg|\\
    &\le \sum_{i=1}^{N_\io}\sum_j\bigg|\int  |\De_y(E_1)|^{1/2}\|d_*^y(\mu_{good,\io,y})\|_{L^2} d\nu_\io^{i,(j)} (y)\bigg|\\
    &\le \sum_{i=1}^{N_\io}\sum_j\bigg(\int  |\De_y(E_1)| d\nu_\io^{i,(j)} (y)\bigg)^{1/2} \bigg(\int \|d_*^y(\mu_{good,\io,y})\|_{L^2}^2 d\nu_\io^{i,(j)} (y)\bigg)^{1/2}\\
    &\le \bigg(\sum_{i=1}^{N_\io}\sum_j\int  |\De_y(E_1)| d\nu_\io^{i,(j)} (y)\bigg)^{1/2} \bigg(\sum_{i=1}^{N_\io}\sum_j\int \|d_*^y(\mu_{good,\io,y})\|_{L^2}^2 d\nu_\io^{i,(j)} (y)\bigg)^{1/2}\\
  (\textup{by~Proposition~}\ref{prop-our-L2-good})  &\le \sup_{y\in E_2}|\De_y(E_1)|^{1/2}\bigg(\sum_{i=1}^{N_\io}\sum_jR_\io^{-\de}\nu_\io^{i,(j)} ([0,1)^2)\bigg)^{1/2}\\
  &\le {R_\io^{-\de/2}}\sup_{y\in E_2}|\De_y(E_1)|^{1/2}.
\end{align*}

\begin{align*}
    III\le \|\mu_\io\|_{L^1} \|\nu_\io^{N_\io+1}\|_{L^1}\lesssim \|\mu\|_{L^1} R_\io^{-\de^3}\le R_\io^{-\de^3}.
\end{align*}

We still need to bound the first term on the right-hand side of \eqref{finalest}.
\begin{align*}
    \bigg|\int\bigg( \int_{\De_y(E_1)}d_*^y(M_0\mu) \bigg)d\nu (y)\bigg|&\le \int |\De_y(E_1)|^{1/2}\|d_*^y(M_0\mu)\|_{L^2}d\nu(y)\\ 
    &\lesssim \sup_{y\in E_2}|\De_y(E_1)|^{1/2} \bigg(\int \|d_*^y(M_0\mu)\|_{L^2}^2d\nu(y)\bigg)^{1/2}\\
    &\lesssim R_0^{O(1)}\sup_{y\in E_2}|\De_y(E_1)|^{1/2}. 
\end{align*} 
The last inequality follows the same argument as in the proof of Proposition \ref{prop-our-L2-good}, but we only need a crude bound.

Now, plugging everything into \eqref{finalest}, we obtain
\begin{equation}
    1\le R_0^{O(1)}\sup_{y\in E_2}|\De_y(E_1)|^{1/2}+\sum_{\io\ge 1}(R_\io^{-\de}+{R_\io^{-\de/2}}\sup_{y\in E_2}|\De_y(E_1)|^{1/2}+R_\io^{-\de^3})+\RD(R_0).
\end{equation}

By choosing $R_0$ large enough, we see that
\[ \sup_{y\in E_2}|\De_y(E_1)|^{1/2}>0. \]
This finishes the proof.

\section{Weighted \texorpdfstring{$L^2$}{} estimate}\label{secweightedL2}

Let $\si_S$ be the surface measure on the unit circle $S$, and we define the extension operator $E:S\to\R^2$ as
\begin{equation}
\label{extension}
    Ef(x)=\int e^{ix\cdot\xi}f(\xi)d\si_S(\xi).
\end{equation}
Consider also its ``broad" variant:

\begin{definition}
\label{def-broad-function-intro}
Let $K\geq 1$ and $A\in\N$.
Let $\Si=\{\si\}$ be a set of 2-overlapping $K^{-1}$-arcs that covers $S$. 
Let $f_\si=f\Id_\si$.
For any $f:S\to\C$ and  $x\in \R^2$, define 
\begin{equation}
    \BR_A Ef(x):=\max_{\substack{\Si'\subset\Si,\\ \#\Si'=A}}\min_{\si\in\Si'}|Ef_\si(x)|.
\end{equation}
\end{definition}

We also recall the classical definition of a Katz-Tao $(R^{-1},s)$-set.
\begin{definition}\label{defclassicalKT}
    Let $R>1$, $s>0$. Let $X\subset B(0,1)$ be a disjoint union of $R^{-1}$-balls. We say $X$ is a Katz-Tao $(R^{-1},s,C)$-set if
    \[ |X\cap B(x,r)|\le CR^{-2} (rR)^s  \]
    for any $B(x,r)$. When $C$ is a fixed large constant, we simply say $X$ is a Katz-Tao $(R^{-1},s)$-set. 
\end{definition}

\begin{lemma}\label{lemwtedL2}
Let $X \subset B_R$ be a union of unit balls such that the $R^{-1}$-dilate of $X$ is a Katz-Tao $(R^{-1},s)$-set. 
Then for all $\e>0$,
\begin{equation}
\label{desired-esti}
    \|Ef\|_{L^2(X)}^2 \leq C_\e R^\e |X|^{\frac{s}{s+1}} \|f\|_2^2.
\end{equation}
Consequently, we obtain
\begin{equation}\label{weightedL2Ef}
    \|Ef\|_{L^2(X)}^2 \leq C_\e R^\e |X|^{\frac{s}{s+1}}R^{-1} \|Ef\|_{L^2(B(0,2R))}^2.
\end{equation}
\end{lemma}

\begin{proof}
The proof relies on the standard broad-narrow reduction and the bilinear restriction theorem.
See \cite{Wu-weighted-L2} for a similar argument.
For a given $\e$ in \eqref{desired-esti}, pick $K=R^{\e^2}$ and $A=8$ in Definition \ref{def-broad-function-intro}.
Clearly, we have
\begin{equation}
\label{broad-narrow}
    \|Ef\|_{L^2(X)}^2 \lesssim  \sum_{\sigma} \int_X |Ef_{\sigma}|^2 + K^{O(1)} \int_X |\text{Br}_8 Ef|^2.
\end{equation}
For a $K^{-1}$-interval $\si\subset[-1,1]$ centered at $\xi_\si$, let $\mathcal{L}_\si$ be the parabolic rescaling
\begin{equation}
\label{parabolic-rescaling}
    \mathcal{L}_\si(x_1,x_2)=(K^{-1}(x_1+2x_2\xi_\si), K^{-2}x_2).
\end{equation}

We distinguish two cases.

\textbf{Case 1.} 
Suppose the first term in \eqref{broad-narrow} dominates. 
Fix $\sigma$. Let $\mathcal{B}_\sigma$ be a family of finite-overlapping $K \times K^2$-rectangles oriented with direction $\sigma$ that forms a cover of $X$. 
By dyadic pigeonholing, there exist $m \in [1,K^3]$ and a set $\mathcal{B}_{\sigma,m} \subset \mathcal{B}_\sigma$ such that the following is true:
\begin{enumerate}
    \item For all $B_\sigma \in \mathcal{B}_{\sigma,m}$, $|X \cap B_\sigma| \sim m$.
    \item We have, since $|Ef_\sigma|$ is essentially constant on each $B_\sigma \in \mathcal{B}_{\sigma,m}$,
\[
    \int_X |Ef_\sigma|^2 \approx \int_{X \cap (\bigcup_{\mathcal{B}_{\sigma,s}}B)} |Ef_\sigma|^2 \lesssim mK^{-3} \int_{\bigcup_{\mathcal{B}_{\sigma,m}}B} |Ef_\sigma|^2.
\]
\end{enumerate}

Thus, either we can easily conclude the lemma, or by dyadic pigeonholing, there exists a union of $K \times K^2$-rectangles $\tilde{X}_\sigma \subset \bigcup_{\mathcal{B}_{\sigma,m}}$ so that
\begin{enumerate}
    \item $\|Ef_\sigma\|_{L^2(B_\sigma)}$  are about the same for all $K \times K^2$-rectangles $B_\sigma \subset \tilde{X}_\sigma$.
    \item We have
\end{enumerate}
\begin{equation}
\label{Ef-sigma}
    \int_X |Ef_\sigma|^2 \lessapprox mK^{-3} \int_{\tilde{X}_\sigma} |Ef_\sigma|^2.
\end{equation}
For any $r \in [1, R/K^2]$, let $Q_r$ be an $rK \times rK^2$-rectangle oriented with direction $\sigma$. Since the $R^{-1}$-dilate of $X$ is a Katz-Tao $(R^{-1},s)$-set, we know that
\[
    |X \cap Q_r| = \sum_{B_r \subset Q_r} |X \cap B_r| \lesssim K (Kr)^s,
\]
where $\{B_r\}$ are finite-overlapping $rK$-balls. 
As a result,
\begin{equation} 
\label{upperbound-ball}
    \# \{ B_\sigma \in \mathcal{B}_{\sigma,m} : B_\sigma \subset Q_r \} \lesssim K^{1+s}r^sm^{-1}.
\end{equation}

By random sampling, there exists a union of $K \times K^2$-rectangles $\tilde{X}_\sigma' \subset \tilde{X}_\sigma$ such that
\begin{enumerate}
    \item $|\tilde{X}_\sigma'| \approx m K^{-(1+s)} |\tilde{X}_\sigma|$.
    \item Recall \eqref{parabolic-rescaling}.
    Inside a unit ball, the $(R/K^2)^{-1}$-dilate of $\mathcal{L}_\sigma(\tilde{X}_\sigma')$ is a Katz-Tao $((R/K^2)^{-1},s)$-set.
\end{enumerate}
Let $g(\xi) = f\big(K^{-1}(\xi - \xi_\sigma)\big)$, where $\xi_\sigma$ is the center of $\sigma$. 
Put the above information back into \eqref{Ef-sigma} so that
\begin{equation}
\label{para-res}
    \int_X |Ef_\sigma|^2 \approx K^{s-2} \int_{\tilde{X}_\sigma'} |Ef_\sigma|^2 = K^{s-1} \int_{\mathcal{L}_\sigma(\tilde{X}_\sigma')} |Eg|^2.
\end{equation}

Since $ |\mathcal{L}_\sigma(\tilde{X}_\sigma')|=K^{-(1+s)}|X|$, apply \eqref{desired-esti} at scale $R/K^2$ to get 
\begin{align}
\label{after-induction}
    \int_{\mathcal{L}_\sigma(\tilde{X}_\sigma')} |Eg|^2&\le C_\varepsilon (R/K^2)^\varepsilon |\mathcal{L}_\sigma(\tilde{X}_\sigma')|^{\frac{s}{s+1}}\|g\|_2^2 \\ \nonumber
    &= K^{-2\varepsilon} K^{1-s}\cdot  C_\varepsilon R^\varepsilon|X|^{\frac{s}{s+1}} \|f_\sigma\|_2^2.
\end{align}
Put \eqref{after-induction} and \eqref{para-res} back to \eqref{broad-narrow} and sum over $\sigma$ to obtain 
\begin{equation*}
    \|Ef\|_{L^2(X)}^2\lessapprox K^{-2\varepsilon} C_\varepsilon R^\varepsilon {|X|^{\frac{s}{s+1}}} \sum_\sigma \|f_\sigma\|_2^2\le C_\varepsilon R^\varepsilon {|X|^{\frac{s}{s+1}}} \|f\|_2^2.
\end{equation*}

\textbf{Case 2.} Suppose the second term in \eqref{broad-narrow} dominates.
Then, by Definition \ref{def-broad-function-intro}, 
\begin{equation}
\nonumber
    \|Ef\|_{L^2(X)}^2\lesssim K^{O(1)} \|\mathrm{Br}_8 Ef\|_{L^2(X)}^2\lesssim R^{O(\varepsilon^2)} \sum_{\substack{\sigma_1, \sigma_2\in\Sigma,\\
    \dist(\sigma_1, \sigma_2)\geq K^{-1}}}\int_X|Ef_{\sigma_1}Ef_{\sigma_2}|.
\end{equation}
Via Cauchy-Schwarz, we have
\begin{equation}
\nonumber
    \sum_{\substack{\sigma_1, \sigma_2\in\Sigma,\\
    \dist(\sigma_1, \sigma_2)\geq K^{-1}}}\int_X|Ef_{\sigma_1}Ef_{\sigma_2}|\lesssim |X|^{1/2}\sum_{\substack{\sigma_1, \sigma_2\in\Sigma,\\
    \dist(\sigma_1, \sigma_2)\geq K^{-1}}}\Big(\int_X|Ef_{\sigma_1}Ef_{\sigma_2}|^2\Big)^{1/2},
\end{equation}
which, by using the bilinear restriction theorem, is bounded above by
\begin{equation}
\nonumber
    K^{O(1)}|X|^{1/2}\sum_{\sigma_1, \sigma_2\in\Sigma}\|f_{\sigma_1}\|_2\|f_{\sigma_2}\|_2\lesssim K^{O(1)}|X|^{1/2}\|f\|_2^2\lesssim K^{O(1)}|X|^{{\frac{s}{s+1}}}\|f\|_2^2.
\end{equation}
Since $K=R^{\e^2}$, the above calculations show that
\[
    \|Ef\|_{L^2(X)}^2\lesssim R^{O(\e^2)}|X|^{{\frac{s}{s+1}}}\|f\|_2^2.  \qedhere
\]

\end{proof}

%***************************************

\newpage

\bibliographystyle{abbrv}
\bibliography{refs}
%***************************************

\end{document}